\documentclass[reqno]{amsart}
\usepackage{latexsym,amsmath,amssymb,amscd, amsthm}
\usepackage{eucal}
\usepackage[utf8]{inputenc}
\usepackage[all]{xy}
\usepackage{enumerate}

\usepackage[T1]{fontenc}
\usepackage
{hyperref} 
\makeatletter
\@addtoreset{figure}{section}
\def\thefigure{\thesection.\@arabic\c@figure}
\def\fps@figure{h,t}
\@addtoreset{table}{bsection}

\def\thetable{\thesection.\@arabic\c@table}
\def\fps@table{h, t}
\@addtoreset{equation}{section}

\makeatletter
\newcommand\@dotsep{4.5}
\def\@tocline#1#2#3#4#5#6#7{\relax
	\ifnum #1>\c@tocdepth 
	\else
	\par \addpenalty\@secpenalty\addvspace{#2}%
	\begingroup \hyphenpenalty\@M
	\@ifempty{#4}{%
		\@tempdima\csname r@tocindent\number#1\endcsname\relax
	}{%
		\@tempdima#4\relax
	}%
	\parindent\z@ \leftskip#3\relax \advance\leftskip\@tempdima\relax
	\rightskip\@pnumwidth plus1em \parfillskip-\@pnumwidth
	#5\leavevmode\hskip-\@tempdima #6\relax
	\leaders\hbox{$\m@th
		\mkern \@dotsep mu\hbox{.}\mkern \@dotsep mu$}\hfill
	\hbox to\@pnumwidth{\@tocpagenum{#7}}\par
	\nobreak
	\endgroup
	\fi}
\makeatother 

\newtheorem{theorem}{Theorem}
\newtheorem*{theorem*}{Theorem}
\newtheorem{corollary}[theorem]{Corollary}
\newtheorem{definition}[theorem]{Definition}
\newtheorem{example}[theorem]{Example}
\newtheorem{lemma}[theorem]{Lemma}
\newtheorem{notation}[theorem]{Notation}
\newtheorem{problem}[theorem]{Problem}
\newtheorem{proposition}[theorem]{Proposition}
\newtheorem{remark}[theorem]{Remark}

\numberwithin{theorem}{section}
\numberwithin{equation}{section}

\renewcommand{\1}{{\bf 1}}

\newcommand{\ad}{{\rm ad}}

\newcommand{\Bun}{\text{{\boldmath{$\mathfrak{B}$}}}}
\newcommand{\Der}{{\rm Der}}

\newcommand{\Cl}{{\rm Cl}}
\newcommand{\ClH}{{\rm ClH}}

\newcommand{\de}{{\rm d}}
\newcommand{\ee}{{\rm e}}
\newcommand{\End}{{\rm End}}

\newcommand{\GL}{{\rm GL}}
\newcommand{\Hom}{{\rm Hom}}
\newcommand{\I}{{\rm I}}

\newcommand{\Ind}{{\rm Ind}}
\newcommand{\ind}{{\rm ind}}

\newcommand{\ie}{{\rm i}}
\newcommand{\Ker}{{\rm Ker}\,}

\newcommand{\nor}{{\rm nor}}

\renewcommand{\O}{{\mathbf{O}}}

\newcommand{\Prim}{{\rm Prim}}

\newcommand{\Rel}{{\boldmath{$\mathfrak{R}$}}}
\newcommand{\RelS}{\text{{\boldmath{$\mathcal{S}$}}}}
\newcommand{\rk}{{\rm rk}}
\newcommand{\Sp}{{\rm Sp}}
\newcommand{\SU}{{\rm SU}}
\newcommand{\spa}{{\rm span}\,}

\newcommand{\supp}{{\rm supp}\,}

\newcommand{\CC}{{\mathbb C}}

\newcommand{\NN}{{\mathbb N}}

\newcommand{\RR}{{\mathbb R}}
\newcommand{\TT}{{\mathbb T}}
\newcommand{\ZZ}{{\mathbb Z}}

\newcommand{\Ac}{{\mathcal A}}
\newcommand{\Bc}{{\mathcal B}}
\newcommand{\Cc}{{\mathcal C}}
\newcommand{\Dc}{{\mathcal D}}
\newcommand{\Ec}{{\mathcal E}}

\newcommand{\Hc}{{\mathcal H}}

\newcommand{\Kc}{{\mathcal K}}

\newcommand{\Oc}{{\mathcal O}}
\newcommand{\Pc}{{\mathcal P}}

\newcommand{\Sc}{{\mathcal S}}
\newcommand{\Tc}{{\mathcal T}}
\newcommand{\Uc}{{\mathcal U}}
\newcommand{\Vc}{{\mathcal V}}

\newcommand{\ag}{{\mathfrak a}}

\newcommand{\dg}{{\mathfrak d}}
\renewcommand{\gg}{{\mathfrak g}}
\newcommand{\hg}{{\mathfrak h}}
\newcommand{\kg}{{\mathfrak k}}

\newcommand{\mg}{{\mathfrak m}}
\renewcommand{\ng}{{\mathfrak n}}

\newcommand{\rg}{{\mathfrak r}}

\newcommand{\zg}{{\mathfrak z}}

\usepackage{mathtools}

\newcommand{\mathsout}[1]
{\bgroup\mathchoice
	{\sbox0{$\displaystyle{#1}$}%
		\usebox0\hspace{-\wd0}%
		\rule[0.5\ht0-0.5\dp0-.5pt]{\wd0}{1pt}}%
	{\sbox0{$\textstyle{#1}$}%
		\usebox0\hspace{-\wd0}%
		\rule[0.5\ht0-0.5\dp0-.5pt]{\wd0}{1pt}}%
	{\sbox0{$\scriptstyle{#1}$}%
		\usebox0\hspace{-\wd0}%
		\rule[0.5\ht0-0.5\dp0-.5pt]{\wd0}{1pt}}%
	{\sbox0{$\scriptscriptstyle{#1}$}%
		\usebox0\hspace{-\wd0}%
		\rule[0.5\ht0-0.5\dp0-.5pt]{\wd0}{1pt}}%
	\egroup}

\title[Strong $C^*$-rigidity of Heisenberg groups]
{Strong $C^*$-rigidity of the Heisenberg groups}

\author{Ingrid Belti\c t\u a}
\author{Daniel Belti\c t\u a}
\address{Institute of Mathematics ``Simion Stoilow'' of the Romanian Academy,
	P.O. Box 1-764, Bucharest, Romania}

\email{Ingrid.Beltita@imar.ro, ingrid.beltita@gmail.com}
\email{Daniel.Beltita@imar.ro, beltita@gmail.com}
\thanks{We acknowledge financial
	support from the Research Grant GAR 2023 (code 114), supported from the Donors’
	Recurrent Fund of the Romanian Academy, managed by the ”PATRIMONIU”
	Foundation.}

\begin{document}
\parskip4pt


\begin{abstract}
We prove a strong rigidity property of the Heisenberg groups, that is, they can be distinguished from any other 1-connected Lie groups via their unitary dual spaces, 
in particular via the Morita equivalence class of their group $C^*$-algebras. 
\end{abstract}
	
\maketitle


\section{Introduction}

Representation theory of a non-commutative Lie group can be modelled by objects with quite diverse flavours, such as the unitary dual space, the group $C^*$-algebra, or the universal enveloping algebra of the corresponding Lie algebra. 
A comprehensive discussion in this connection can be found in the recent book \cite{BdlH2020}. 
It is natural, and yet often nontrivial, that key properties of the Lie group under consideration are encoded by the algebraic-topological objects mentioned above. 
Several instances of that phenomenon can be found in the earlier literature. 
See for instance the study of the unitary dual of a Lie group versus its underlying discrete group in \cite{BV93}. 
As another example, related to the question whether a 1-connected Lie group can be recovered from the universal enveloping algebra of its Lie algebra, 
we mention the recent solution for nilpotent Lie groups in \cite{CPRW24}, 
as well as the results in the solvable case in \cite{VSU22}. 

In \cite{BB-IEOT} and \cite{BB-JTA}, we approached the analogous problem, with universal enveloping algebras replaced by group $C^*$-algebras. 
Thus, we provided families of mutually non-isomorphic solvable Lie groups with isomorphic $C^*$-algebras, but also quite a few examples of nilpotent Lie groups that are uniquely determined by their group $C^*$-algebras within the class of nilpotent Lie groups. 
A related problem, namely reconstruction of Lie groupoids from their $C^*$-algebras was  also recently studied in \cite{CRST21}, \cite{CDRS25}, and \cite{DS23}. 
Moreover, we have recently proved in \cite{BB_RIMS} that, within the class of all  1-connected Lie groups, every Heisenberg group is uniquely determined by the isomorphism class of its group $C^*$-algebra. 
We recall that for an arbitrary integer $n\ge 1$ the Heisenberg algebra $\hg_n$ is 
the Lie algebra with a basis $X_1,\dots,X_n,Y_1,\dots,Y_n,Z$ satisfying the canonical commutation relations 
$$[X_j,Y_k]=\delta_{jk}Z\text{ for }j,k=1,\dots,n$$
where $\delta_{jk}=1$ if $j=k$ and $\delta_{jk}=0$ otherwise, while other commutation relations among elements of the basis are trivial. 
The Heisenberg group $H_n$ is the 1-connected Lie group whose Lie algebra is $\hg_n$. 
These Lie groups have played a unique role in the development of very diverse research areas.  
For instance, one of the primary motivations of  bivariant K-theory [Ka79,Vo81], was the $C^*$-algebra extension 
$$0\to \Cc_0(\RR^\times)\otimes\Kc\to C^*(H_1)\to \Cc_0(\RR^2)\to 0$$ 
involving the $C^*$-algebra of the 3-dimensional Heisenberg $H_1$, 
which also motivated our present approach, 
cf.  Remark~\ref{partition_rem} below.

In the present paper, we obtain a sharp improvement of the above result 
from~\cite{BB_RIMS}, replacing the group $C^*$-algebra by the unitary dual space with its natural topology. 
Here is our main result.

\begin{theorem}	\label{dual_main}
	If $G$ is a 1-connected Lie group and $H$ is a Heisenberg group with a homeomorphism $\widehat{G}\simeq \widehat{H}$, then there exists a Lie group isomorphism $G\simeq H$. 
\end{theorem}

\begin{corollary}	
 The Heisenberg groups can be distinguished from any other\break 
 1-connected Lie groups via the Morita equivalence class of their group $C^*$-algebras. 
\end{corollary}

Such a strong rigidity property of Heisenberg groups is remarkable since the unitary dual space of a Lie group encodes less information than the universal enveloping algebra or the group $C^*$-algebra, as the following examples show:

\begin
{customthm}{A}\label{A}
\normalfont 
If the unitary dual spaces of two 1-connected Lie groups are homeomorphic, 
then these groups may not be isomorphic even under the additional assumption that they are nilpotent (Example~\ref{E14}).  
If they are not isomorphic then, according to \cite[Cor. 0.12]{CPRW24}, the corresponding universal enveloping algebras are nonisomorphic, too. 
\end
{customthm}

\begin
{customthm}{B}\label{B}
	\normalfont  
Similarly, if the unitary dual spaces of two 1-connected Lie groups are homeomorphic, then their corresponding group $C^*$-algebras may be nonisomorphic. 
For instance, the Lie groups $\SU(2)$ and $\SU(3)$ 
 are compact, hence their unitary dual spaces are countable discrete topological spaces, in particular homeomorphic to each other. 
On the other hand, $\SU(2)$ has exactly one irreducible representation on an $n$-dimensional complex Hilbert space for every $n\ge 1$, 
and this property is shared by its group $C^*$-algebra. 
However, there exist several nonequivalent irreducible representations of $\SU(3)$ on an $n$-dimensional complex Hilbert space for some $n\ge 1$, 
while there are no irreducible representations for other values $n\ge 1$, 
and this property is again shared by its group $C^*$-algebra. 
See e.g., \cite{Ha15}.
Therefore the group $C^*$-algebras of $\SU(2)$ and $\SU(3)$ are not isomorphic to each other. 
\end
{customthm}

The object of the present paper is to give the proof of Theorem~\ref{dual_main}. 
Using \cite{MoRo76}, \cite{Pu78}, and one of our previous results in \cite{BB_RIMS}, which shows that the unitary dual of the Heisenberg group cannot be written as the Cartesian product of two non-trivial sets, we reduce the problem to the case where $G$ has co-compact radical. 
Thus, we  prove Theorem~\ref{dual_main} in two steps that can be briefly described as follows:
\begin{enumerate}[{Step 1:}]
	\item Under the hypothesis of Theorem~\ref{dual_main}, if the Lie group $G$ 
	has co-compact radical, then  the group $G$ is actually solvable  
	(Proposition~\ref{cocomp}). 
	\item Theorem~\ref{dual_main} holds true under the additional assumption that $G$ is a solvable Lie group 
	(Proposition~\ref{P13}).
\end{enumerate}
There is a common idea in both steps above, namely, the `size' of  limit sets of certain sequences in the unitary spectrum gives information on the group under consideration.
Nevertheless, these two steps are essentially independent of each other and require rather different panels of techniques. 
For this reason we may go in the natural direction of increasing the degree of generality. 
Specifically, we first study solvable Lie groups, thereby performing Step 2 in Section~\ref{sect:solvtypeI}. Here we use our
characterization of nilpotent Lie groups among the solvable Lie groups in terms of topological properties of their coadjoint orbits, cf. \cite[Thm. 1.1]{BB_RIMS}.  Nevertheless, as the unitary dual of a Lie group encodes less information than the group $C^*$-algebra, finer tools are also needed.
Then, we put to work the powerful machinery of induced representation to study the Lie groups with co-compact radical in Section~\ref{sect:cocomp}, 
which realizes Step 1 above. 

In some more detail, the contents of this paper are as follows. 
Section~\ref{sect: preltop} describes certain topological spaces that are central to our investigation. 
Results on induced representations of locally compact groups, which are needed in the next sections and could not be located in the literature, are proved in Section~\ref{subsect:ind}.
In Section~\ref{sect:prelsolv} we recall the basic setting of the Puk\'anszky correspondence in representation theory of solvable Lie groups (cf. \cite{Pu71}, \cite{Pu73}) and establish some additional auxiliary facts for later use. 
Section~\ref{sect:solvtypeI} contains the proof of the special case of Theorem~\ref{dual_main} for solvable Lie groups (Proposition~\ref{P13}), 
that is, Step~2 above.  
There is also a class of examples that show that even that special case fails to be true with the Heisenberg group replaced by an arbitrary nilpotent Lie group. 
More specifically, we provide a family of mutually nonisomorphic 7-dimensional nilpotent Lie groups whose unitary dual spaces are mutually homeomorphic (Example~\ref{E14}). 
Section~\ref{sect:cocomp} achieves Step~1 above (Proposition~\ref{cocomp}), 
where a key role is played a deformation technique for induced representations that is developed in Section~\ref{subsect:ind}. 
Finally, Section~\ref{sect:proof} gives the proof of Theorem~\ref{dual_main} along with some related open questions. 
\subsection*{General notation and terminology} \hfill\\
We denote  $\TT:=\{z\in\CC\mid \vert z\vert=1\}$ and  $\RR^\times:=\RR\setminus\{0\}$.
Throughout this paper, `1-connected' means connected and simply connected. 
Every 1-connected Lie group is denoted by an upper case Roman letter, 
and its Lie algebra is denoted by the corresponding lower case Gothic letter. 
For any Lie algebra $\gg$ with its linear dual space $\gg^*$ we denote by $\langle\cdot,\cdot\rangle\colon\gg^*\times\gg\to\RR$ the corresponding duality pairing. 
We often denote the group actions simply by juxtaposition, 
in particular this is the case for the coadjoint action $G\times\gg^*\to\gg^*$, $(g,\xi)\mapsto g\xi$.

\section{Heisenberg partitions}
\label{sect: preltop}

In this section
we study topological spaces $X$ that admit a homeomorphism $\psi\colon X\to\widehat{H}$ 
onto the unitary dual space of a Heisenberg group~$H$. 
As a first step towards proving the dual rigidity property of Heisenberg groups (Theorem~\ref{dual_main}), we show in Lemma~\ref{L0} that, if $Z$ is the center of $H$ 
and the unitary dual space $\widehat{H/Z}$ is embedded as the closed subset of $\widehat{H}$ corresponding to the 1-dimensional unitary representations of $H$, then the subset $\psi^{-1}(\widehat{H/Z})\subseteq X$ is independent on the homeomorphism $\psi$. 
Equivalently, every homeomorphism of $\widehat{H}$ preserves the partition $\widehat{H}=(\widehat{H}\setminus \widehat{H/Z})\sqcup\widehat{H/Z}$. 
That property is not trivial since not every homeomorphism of $\widehat{H}$ is defined by an automorphism of~$H$.

The following lemma is suggested by the description of the topology of the unitary duals of Heisenberg groups, 
first given in \cite[Prop. 1]{Di60} for lower dimensional groups.  
We need this result in Example~\ref{E14} for the space $Y$ equal to the unitary dual of a certain 6-dimensional nilpotent Lie group, 
and for $Y=\RR^{2n}$ in Lemma~\ref{L0}. 

\begin{lemma}
	\label{L00}
	Let $X$ and $Y$ be two topological spaces endowed with injective maps $\psi_1\colon\RR^\times\to X$ and $\psi_2\colon Y\to X$ 
	satisfying the following conditions: 
	\begin{enumerate}[{\rm(a)}]
		\item\label{L00_item_a} The subset $\Gamma_1:=\psi_1(\RR^\times)\subseteq X$ is open and the mapping  $\psi_1\colon\RR^\times\to\Gamma_1$ is a homeomorphism. 
		\item\label{L00_item_b} For the subset $\Gamma_2:=\psi_2(Y)\subseteq X$ 
		we have $\Gamma_2=X\setminus\Gamma_1$ and the mapping   
		$\psi_2\colon Y\to\Gamma_2$ is a homeomorphism. 
		\item\label{L00_item_c} For every subset $A\subseteq\RR^\times$ we have
		\begin{equation*}
			\text{$0\in\RR$ is an accumulation point for $A$}
			\iff \Gamma_2\cap \overline{\psi_1(A)}\ne\emptyset \iff \Gamma_2\subseteq\overline{\psi_1(A)}.
		\end{equation*}
	\end{enumerate}
	 Then a subset $F\subseteq X$ is closed if and only if it satisfies the following conditions: 
	 \begin{enumerate}[{\rm(i)}]
	 	\item\label{L00_item1} $\psi_1^{-1}(F)$ is closed in $\RR^\times$. 
	 	\item\label{L00_item2} $\psi_2^{-1}(F)$ is closed in $Y$. 
	 	\item\label{L00_item3} if $0\in\RR$ is an accumulation point for $\psi_1^{-1}(F)$, then $\Gamma_2\subseteq F$. 
	 \end{enumerate}
 Moreover, the  topology of $X$ is $T_1$ if and only if the topology of $Y$ is $T_1$. 
\end{lemma}

\begin{proof}
If $F\subseteq X$ is closed, then $F\cap\Gamma_j$ is relatively closed in $\Gamma_j$ for $j=1,2$. 
Since the mappings $\psi_1\colon\RR^\times\to\Gamma_1$ and 	$\psi_2\colon Y\to\Gamma_2$ are homeomorphisms, we obtain \eqref{L00_item1}--\eqref{L00_item2}. 
If  $0\in\RR$ is an accumulation point for $A:=\psi_1^{-1}(F)\subseteq \RR^\times$, 
then by \eqref{L00_item_c} we have 
$\Gamma_2\subseteq \overline{\psi_1(A)}\subseteq \overline{F}=F$. 

Conversely, let $F\subseteq X$  satisfying \eqref{L00_item1}, \eqref{L00_item2}, and \eqref{L00_item3}. 
We will prove that $F$ is closed in $X$. 
To this end we denote  $A:=\psi_1^{-1}(F)\subseteq \RR^\times$ and $B:=\psi_2^{-1}(F)\subseteq Y$, 
hence $F\cap \Gamma_1=\psi_1(A)$ and $F\cap\Gamma_2=\psi_2(B)$. 
Here the subset $B\subseteq Y$ is closed by \eqref{L00_item2} hence, 
since 	$\psi_2\colon Y\to\Gamma_2$ is a homeomorphism, it follows that $F\cap\Gamma_2$ is relatively closed in $\Gamma_2$. 
As the subset $\Gamma_2\subseteq X$ is closed by \eqref{L00_item1}--\eqref{L00_item2}, 
it then follows that $F\cap\Gamma_2$ is closed in $X$. 

Furthermore, since $X=\Gamma_1\sqcup\Gamma_2$, we have $F=(F\cap \Gamma_1)\sqcup(F\cap \Gamma_2)$, 
hence 
\begin{equation*}
\overline{F}=\overline{F\cap \Gamma_1}\cup\overline{F\cap \Gamma_2} 
=\overline{\psi_1(A)}\cup\overline{F\cap \Gamma_2} 
=\overline{\psi_1(A)}\cup(F\cap \Gamma_2) 
\end{equation*}
since we established above that $F\cap\Gamma_2$ is closed in $X$. 
Moreover, the subset $A\subseteq\RR^\times$ is closed by \eqref{L00_item1} 
and the mapping $\psi_1\colon\RR^\times \to\Gamma_1$ is a homeomorphism, 
hence $\psi_1(A)$ is relatively closed in $\Gamma_1$, 
which is equivalent to $\psi_1(A)=\overline{\psi_1(A)}\cap\Gamma_1$. 
Since $X=\Gamma_1\sqcup\Gamma_2$, we further obtain 
\begin{equation}
	\label{L00_proof_eq1}
	\overline{\psi_1(A)}=\psi_1(A)\sqcup(\overline{\psi_1(A)}\cap\Gamma_2).
\end{equation} 
Then the above equality implies 
\begin{equation}
\label{L00_proof_eq2}
	\overline{F}=\psi_1(A)\cup ((F\cup \overline{\psi_1(A)})\cap \Gamma_2). 
\end{equation}
Now, if $0\in\RR$ is an accumulation point for $A$  
then, by \eqref{L00_item3}, we have $\Gamma_2\subseteq F$, 
hence~\eqref{L00_proof_eq2} implies $\overline{F}=\psi_1(A)\cup \Gamma_2\subseteq F$, 
which shows that $F$ is closed in $X$. 
On the other hand, if $0\in\RR$ is not an accumulation point for $A$, then \eqref{L00_item_c} implies $\Gamma_2\cap\overline{\psi_1(A)}=\emptyset$, 
hence \eqref{L00_proof_eq2} shows that $\overline{F}=\psi_1(A)\cup (F\cap \Gamma_2)\subseteq F$,
and we obtain again that  $F$ is closed in $X$. 

Finally, if the topology of $X$ is $T_1$ then the topology of $Y$ is $T_1$ since $Y$ is homeomorphic to the closed subset $\Gamma_2$ of $X$. 
Conversely, if the topology of $Y$ is $T_1$, let $x\in X=\Gamma_1\sqcup\Gamma_2$  arbitrary. 
If $x\in\Gamma_2$, then the singleton $\{x\}$ is a closed subset of $\Gamma_2$ since $Y$ is $T_1$ and is homeomorphic to $\Gamma_2$. 
Since $\Gamma_2$ is closed in $X$, it follows that $\{x\}$ is closed in $X$. 
On the other hand, if $x\in\Gamma_1$, we define $a:=\psi_1^{-1}(x)\in\RR^\times$ 
and $A:=\{a\}\subseteq\RR^\times$. 
Then $0\in\RR$ is not an accumulation point for $A$, then \eqref{L00_item_c} implies $\Gamma_2\cap\overline{\psi_1(A)}=\emptyset$, 
hence $\overline{\psi_1(A)}\subseteq\Gamma_1$, 
that is, $\overline{\{x\}}\subseteq\Gamma_1$. 
Since the relative topology of $\Gamma_1$ is Hausdorff, we obtain $\overline{\{x\}}=\{x\}$, and this completes the proof. 
\end{proof}

\begin{remark}
\label{L00_rem}
\normalfont
In Lemma~\ref{L00}, the description of the closed subsets of $X$ shows that the topology of $X$ is uniquely determined by the data $(Y,\psi_1,\psi_2)$ subject to the conditions \eqref{L00_item_a}, \eqref{L00_item_b}, and \eqref{L00_item_c}.  
\end{remark}

\begin{notation}
	\normalfont 
	For an arbitrary topological space $X$ we denote by $\ClH(X)$ the set of all closed subsets of $X$ which are connected and Hausdorff with respect to the relative topology. 
	The set $\ClH(X)$ is regarded as a partially ordered set with respect to the set inclusion. 
\end{notation}

We now specialize the setting of Lemma~\ref{L00} for $Y=\RR^{2n}$. 

\begin{lemma}
\label{L0}
For every topological space $X$ there exists at most one partition $X=\Gamma_1\sqcup\Gamma_2$ with the following properties: 
\begin{enumerate}[{\rm(i)}]
	\item\label{L0_item1} $\Gamma_1\subseteq X$ is an open subset and there exists a homeomorphism $\psi_1\colon\RR^\times\to\Gamma_1$. 
	\item\label{L0_item2} There exists a homeomorphism $\psi_2\colon\RR^{2n}\to\Gamma_2$. 
	\item\label{L0_item3} For every subset $A\subseteq\RR^\times$ we have
	\begin{equation*}
		\text{$0\in\RR$ is an accumulation point for $A$}
		\iff \Gamma_2\cap \overline{\psi_1(A)}\ne\emptyset \iff \Gamma_2\subseteq\overline{\psi_1(A)}.
	\end{equation*}
\end{enumerate}
Moreover, for a partition as above, $\Gamma_2$ is the only maximal element of $\ClH(X)$. 
\end{lemma}

\begin{proof}
Let $X=\Gamma_1\sqcup\Gamma_2$ be a partition as in the statement. 
	
Step 1: We prove that $\Gamma_2$ is a maximal element of $\ClH(X)$. 

It is clear that $\Gamma_2\in\ClH(X)$. 
Assume that there exists 
$\Gamma\in\ClH(X)$ with $\Gamma_2\subsetneqq\Gamma$. 
Then necessarily $\Gamma\cap\Gamma_1\ne\emptyset$. 
Let us denote $A:=\psi_1^{-1}(\Gamma\cap\Gamma_1)\subseteq\RR^\times$. 
Then $0\in\RR$ is not an accumulation point of $A$.
Indeed, if we assume that
$0\in\RR$ is an accumulation point of $A$, then, by \eqref{L0_item3}, 
the relative topology of $\Gamma$ is not Hausdorff, which is a contradiction with  $\Gamma\in\ClH(X)$.

Since $\Gamma\subseteq X$ is closed and $A=\psi_1^{-1}(\Gamma\cap\Gamma_1)$, it follows that the subset $A\subseteq \RR^\times$ is closed. 
Since $0\in A$ is not an accumulation point of $A$, 
it follows by the description of the closed subsets of $X$ in Lemma~\ref{L00}
that $\psi_1(A)$ ($=\Gamma\cap\Gamma_1$)
is closed in $X$ 
hence the subset $\Gamma\cap\Gamma_1\subseteq\Gamma$ is closed with respect to the relative topology of $\Gamma$. 
Then the equality $\Gamma=(\Gamma\cap\Gamma_1)\sqcup(\Gamma\cap\Gamma_2)$ is a partition into closed subsets with respect to the relative topology of $\Gamma$.   
On the other hand, $\Gamma$ is connected since  $\Gamma\in\ClH(X)$. 
Since  $\Gamma\cap\Gamma_1\ne\emptyset$, we then obtain  $\Gamma\cap\Gamma_1=\Gamma$, 
that is, $\Gamma\subseteq\Gamma_1$.  
But this is a contradiction with the assumption $\Gamma_2\subset\Gamma$.

Step 2: We prove that $\Gamma_2$ is the only maximal element of $\ClH(X)$. 

Let us assume that $\Gamma\in\ClH(X)$ is a maximal element with $\Gamma\ne\Gamma_2$. 
If $\Gamma\cap\Gamma_1=\emptyset$, then $\Gamma\subseteq\Gamma_2$, 
which is a contradiction with the maximality property of $\Gamma$. 
Thus necessarily $\Gamma\cap\Gamma_1\ne\emptyset$. 
Let us denote $A:=\psi_1^{-1}(\Gamma\cap\Gamma_1)\subseteq\RR^\times$ as in Step 1. 
If $0\in\RR$ is an accumulation point of $A$, then, by \eqref{L0_item3}, 
 $\Gamma_2\subseteq\overline{\psi_1(A)}\subseteq\Gamma$, 
 which is again a contradiction as above. 
 Therefore $0\in\RR$ is not an accumulation point of $A$. 
 But $\RR^\times$ has no maximal closed connected subset that does not accumulate to $0\in\RR$, so we obtain a contradiction with the maximality property of $\Gamma$.
 
 Step 3: Let us assume that $X=\Gamma_1'\sqcup\Gamma_2'$ is another partition as in the statement. 
 
By Steps 1--2 above, we obtain $\Gamma_2=\Gamma_2'$, and then $\Gamma_1=\Gamma_1'$ as well. 
\end{proof}

\begin{definition}
\label{partition}
\normalfont
A \emph{Heisenberg partition} of a topological space $X$ 
is a partition $X=\Gamma_1\sqcup\Gamma_2$  
satisfying the conditions \eqref{L0_item1}--\eqref{L0_item3} 
in Lemma~\ref{L0}. 
\end{definition}

\begin{remark}
\label{partition_rem}
\normalfont
In connection with Definition~\ref{partition}, we note the following.
\begin{enumerate}[{\rm(i)}]
	\item\label{partition_rem_item1} The name `Heisenberg partition' is motivated by the  natural one-to-one correspondence between pairs $(\psi_1,\psi_2)$ defining Heisenberg partitions of $X$ and homeomorphisms 
	$\psi\colon \widehat{H}\to X$, where $H$ is the Heisenberg group with $\dim H=2n+1$. 
	More specifically, we recall the canonical parameterization $\widehat{H}=\RR^\times\sqcup\RR^{2n}$ (see e.g., \cite{Di60} and \cite{LuTu11}), 
	and for any mapping $\psi\colon \widehat{H}\to X$ let us define  
	$\psi_1:=\psi\vert_{\RR^\times}$ and $\psi_2:=\psi\vert_{\RR^{2n}}$. 
	Then Remark~\ref{L00_rem} shows that the pair $(\psi_1,\psi_2)$ defines a Heisenberg partition of $X$ if and only if the mapping $\psi\colon \widehat{H}\to X$ is bijective and has the property that for every $F\subseteq \widehat{H}$ 
	we have $F$ closed in $\widehat{H}$ if and only if $\psi(F)$ is closed in $X$, 
	that is, if and only if $\psi$ is a homeomorphism. 
	\item\label{partition_rem_item2}  Every finite subset of $X$ is closed, i.e., the topological space $X$ is $T_1$, by Lemma~\ref{L00}. 
	\item\label{partition_rem_item3}  The integer $n\ge 1$ intrinsically depends on $X$, by Brouwer's theorem on invariance of dimension.
\end{enumerate}
\end{remark}
 
\begin{lemma}
\label{L4}
If $X=\Gamma_1\sqcup\Gamma_2$ is a Heisenberg partition of a topological space, 
then for every continuous injective map $\gamma\colon\TT\to X$ we have $\gamma(\TT)\subseteq \Gamma_2$. 
\end{lemma}

\begin{proof}
We keep the notation in Lemma~\ref{L0}.
Let 
$\Gamma_1^{\pm}:=\psi_1(\RR_{\pm})$ hence 
$\Gamma_1=\Gamma_1^+\sqcup\Gamma_1^-$ is the decomposition of $\Gamma_1$ into connected components. 
Since $\gamma$ is continuous and $\TT$ is compact, it follows that $\gamma(\TT)\subseteq X$ is compact. 
Then, by \cite[Prop. 5.3]{BB17}, we obtain the following properties: 
\begin{enumerate}[1.]
	\item\label{L4_proof_item1} 
	The subset $\gamma(\TT)\cap\Gamma_2\subseteq\Gamma_2$ is compact. 
	\item\label{L4_proof_item2} 
	The subset $\gamma(\TT)\cap\Gamma_1\subseteq\Gamma_1$ is closed and bounded. 
	\item\label{L4_proof_item3} 
	If $0\in\RR$ is an accumulation point of $\psi_1^{-1}(\gamma(\TT)\cap\Gamma_1)$, 
	then $\gamma(\TT)\cap\Gamma_2\ne\emptyset$. 
\end{enumerate}
We now argue by contradiction. 
Let us assume $\gamma(\TT)\cap\Gamma_1\ne\emptyset$. 
Then we may assume without loss of generality $\gamma(\TT)\cap\Gamma_1^+\ne\emptyset$. 
By property \ref{L4_proof_item2} above, there exists $z_0\in\TT$ with $$\psi_1^{-1}(\gamma(z_0))=\sup\psi_1^{-1}(\gamma(\TT)\cap\Gamma_1^+).$$
Again without loss of generality we may assume $z_0=1$. 
Since $\Gamma_1^+$ is an open subset of the open subset $\Gamma_1$ of $X$, 
it follows that the subset $\Gamma_1^+\subseteq X$ is open. 
Since $\gamma\colon\TT\to X$ is continuous and $\gamma(1)\in\Gamma_1^+$, 
there exists $\varepsilon\in(0,\pi)$ with $\gamma_0(I_\varepsilon)\subseteq \Gamma_1^+$, 
where $I_\varepsilon:=(-\varepsilon,\varepsilon)\subseteq\RR$ and $\gamma_0\colon I_\varepsilon\to X$, $\gamma_0(t):=\gamma(\ee^{\ie t})$. 
The property $\gamma_0(I_\varepsilon)\subseteq \Gamma_1^+$ allows us to consider the function $\psi_1^{-1}\circ\gamma_0\colon I_\varepsilon\to(0,\infty)$, which is continuous and injective, hence either increasing or decreasing. 
On the other hand, by the choice of $z_0\in\TT$ above and the assumption $z_0=1\in\TT$ we have $(\psi_1^{-1}\circ\gamma_0)(0)=\sup(\psi_1^{-1}\circ\gamma_0)(I_\varepsilon)$, 
which is impossible if $\psi_1^{-1}\circ\gamma_0\colon I_\varepsilon\to(0,\infty)$  is  either increasing or decreasing. 
This contradiction shows that we cannot have $\gamma(\TT)\cap\Gamma_1\ne\emptyset$, 
therefore $\gamma(\TT)\subseteq \Gamma_2$. 
\end{proof}

\section{Results on induced representations}
\label{subsect:ind}

In the present section we establish, for later use, 
several properties for which we did not find suitable references in the literature.

We first recall a few notions and notation to be used throughout the paper. 
Our general reference for $C^*$-algebras and basic representation theory is \cite{Di60}. 
For every locally compact group $\Gamma$ we denote by $C^*(\Gamma)$ its corresponding group $C^*$-algebra. 
A unitary representation of $\Gamma$ is a homomorphism of topological groups $\pi\colon \Gamma\to\Uc(\Hc)$,   
where $\Uc(\Hc)$ is the group of unitary operators on the complex Hilbert space $\Hc$, regarded as a topological group with its strong operator topology. 
Then $\pi$ gives rise to a nondegenerate $*$-representation $\widetilde{\pi}\colon C^*(\Gamma)\to\Bc(\Hc)$ satisfying $\widetilde{\pi}(f):=\int_\Gamma f(\gamma)\pi(\gamma)$ for every $f\in L^1(\Gamma)$. 
Since $\pi$ is a group morphism, its kernel $\Ker\pi:=\{\gamma\in\Gamma\mid\pi(\gamma)=\1\}$ is a closed normal subgroup of $\Gamma$. 
Similarly, since $\widetilde{\pi}$ is a morphism of Banach $*$-algebras, its kernel $\ker\widetilde{\pi}:=\{a\in C^*(\Gamma)\mid \widetilde{\pi}(a)=0\}$ is a closed two-sided ideal of $C^*(\Gamma)$. 
For the sake of simplicity, we usually write $\pi$ instead of $\widetilde{\pi}$, 
but we keep the differing notation $\Ker\pi$ for the kernel of the group morphism $\pi$ and $\ker\pi$ for the kernel of the $*$-algebra morphism $\pi$ ($=\widetilde{\pi}$).

If $\Ac$ is a $*$-algebra, then we denote by $\widehat{\Ac}$ its set of unitary equivalence classes~$[\pi]$ of irreducible $*$-representations $\pi\colon\Ac\to\Bc(\Hc_\pi)$ and by $\Prim(\Ac):=\{\ker\pi\mid [\pi]\in\widehat{\Ac}\}$ the set of all primitive ideals of $\Ac$. 
If $\Sc$ and $\Tc$ are two sets of $*$-representations of $\Ac$,  $\Tc$ is said to be weakly contained in $\Sc$ if 
$\bigcap\limits_{S\in\Sc}\ker S\subseteq \bigcap\limits_{T\in\Tc}\ker T$, 
and we write $\Tc\preceq\Sc$. 
If both $\Tc\preceq\Sc$ and $\Sc\preceq\Tc$, then $\Sc$ and $\Tc$ are said to be weakly equivalent. 
The notion of weak containment actually depends only on unitary equivalence classes of $*$-representations of $\Ac$, so we may apply it to subsets of $\widehat{\Ac}$. 
For a subset $\Sc\subseteq\widehat{\Ac}$, the set of elements of $\widehat{A}$ that are weakly contained in $\Sc$ is exactly the closure of $\Sc$ in $\widehat{\Ac}$. 
Here $\widehat{\Ac}$ is endowed with the pull-back of the hull-kernel topology via the surjective map $\ker\colon\widehat{\Ac}\to\Prim(\Ac)$. 
For every $*$-representation $T$ of $\Ac$ there exists a unique closed subset of $\widehat{\Ac}$ that is weakly equivalent to $T$, 
namely 
$$\Sp(T):=\{[\rho]\in\widehat{\Ac}\mid \rho\preceq T\},$$
hence $\ker T=\bigcap\limits_{[\pi]\in\Sp(T)}\ker\pi$. 

For every topological space $X$ we denote by $\Cl(X)$ the space of all closed subsets of $X$ endowed with its upper topology, cf. \cite[Def. 2.1]{BB-LMS}, called the inner  topology in \cite[\S 2]{Fe62}.  
If we denote by $\Tc^\sim(\Ac)$ the set of all unitary equivalence classes of nondegenerate $*$-representations of $\Ac$ on separable complex Hilbert spaces, 
then we have the natural map 
\begin{equation}
\label{general_eq1}
\Sp\colon\Tc^\sim(\Ac)\to\Cl(\Prim(\Ac)),
\end{equation} 
and the corresponding pull-back of the inner topology of $\Cl(\Prim(\Ac))$ is called the inner hull-kernel topology of~$\Tc^\sim(\Ac)$. 

When $\Ac=C^*(\Gamma)$ for a locally compact group $\Gamma$, the above terminology on $*$-representations of $\Ac$ carries over to unitary representations of $\Gamma$. 
See \cite{Fe60}, \cite{Fe62}, and \cite{Di64}.

In the following sections we need certain continuity properties of basic operations on representations of locally compact groups, particularly restriction or induction of representations, for which we refer to \cite{Ma58}, \cite{Fe60} and \cite{Fe62}. 
	Let $G$  be a locally compact group with a closed subgroup $K\subseteq G$, 
	with their corresponding modular functions $\Delta_G\colon G\to\RR^\times_+$ and $\Delta_K\colon K\to\RR^\times_+$. 
	There always exist a continuous function $\rho_K\colon G\to\RR^\times_+$ 
	satisfying $\rho_K(gk)=(\Delta_K(k)/\Delta_G(k))\rho_K(g)$ for all $g\in G$ and $k\in K$, 
	and a positive Radon measure $\nu_K$ on $K$ such that
	for every function $f\in\Cc_c(G/K)$ 
	$$\int_{G/K}f(g^{-1}xK)\de\nu_K(xK)=\int_{G/K}\lambda_g(xK)f(xK)\de\nu_K(xK) 
	\text{ for }g\in G,$$
	where $\lambda_g(xK):=\rho_K(g^{-1}x)/\rho_K(x)$ for $g,x\in G$. 
	
	For any representation $\tau\colon K\to\Uc(\Hc_\tau)$, we denote by $\Ind_K^G(\tau)\colon G\to\Uc(\Hc)$ the corresponding induced representation (in the sense of Mackey \cite{Ma58}). 
	We recall that,  if we denote by $\Hc^0$ the space of continuous functions  $\varphi\in\Cc(G,\Hc_\tau)$ such that
	\begin{eqnarray*}
		& \varphi(gk)  =\tau^{-1}(k)\varphi(g) \quad (\forall k\in K,g\in G) \\
		& g K \mapsto \Vert \varphi(g)\Vert \text{ has compact support in }  G/K, 
	\end{eqnarray*} 
	then $\Hc^0$ is dense in the representation Hilbert space $\Hc$ of the induced representation $\Ind_K^G(\tau)=:\pi$, 
	and $(\pi(g)\varphi)(x)=\lambda_g(xK)^{1/2}\varphi(g^{-1}x)$ for all $g,x\in G$ and $\varphi\in\Hc^0$. 
	
	\begin{remark}
	\label{unimod}
	\normalfont 
	There exists a (nonzero) $G$-invariant Radon measure $\nu_K$ on the homogeneous space $G/K$ if and only if the function $\rho_K$ is constant, 
	and then we take $\rho_K(g)=1$ for all $g\in G$. 
	This is the case   
	if for instance both $G$ and $K$ are unimodular \cite[Ch. VII, \S2, no. 6, Cor. 2]{Bo07} or if $K$ is a normal subgroup  \cite[Ch. VII, \S2, no. 7, Prop. 10]{Bo07}.
	\end{remark}

Lemma~\ref{L9} below is used in the proof of Lemma~\ref{L10}. 

\begin{lemma}
	\label{L9}
	Let $G$ be a locally compact group with a closed subgroup $K$ satisfying $[G,G]\subseteq K\subseteq G$. 
	If we denote by $\tau_0\colon K\to\TT$ the trivial representation of $K$ and  define $\pi:=\Ind_K^G(\tau_0)\colon G\to\Bc(L^2(G/K))$, 
	then $\pi$ is a factor representation if and only if $G=K$. 
\end{lemma}

\begin{proof}
	We first note that $K$ is a normal subgroup of $G$
	hence, by 
	Remark~\ref{unimod}, 
	there is a $G$-invariant measure on $G/K$, namely the Haar measure of the locally compact abelian group $G/K$. 
	
	On the other hand, we have $(\pi(g)\varphi)(xK)=\varphi(g^{-1}xK)$ for all $g,x\in G$ and $\varphi\in L^2(G/K)$. 
	Also, for arbitrary $g_1,g_2,x\in G$, denoting $k:=g_1^{-1}g_2^{-1}g_1g_2\in K$, we have 
	$$g_1g_2xK=g_2g_1 kxK=g_2g_1x(x^{-1}kxk^{-1})K=g_2g_1xK. $$
	This implies $\pi(g_1)\pi(g_2)=\pi(g_2)\pi(g_1)$ for arbitrary $g_1,g_2\in G$, 
	hence the von Neumann algebra $\pi(G)''\subseteq\Bc(L^2(G/K))$ generated by $\pi(G)$ is commutative. 
	
	Now, if we assume that $\pi(G)''$ is a factor, then necessarily $\pi(G)''=\CC\1$, 
	hence there exists $\chi\in\Hom(G,\TT)$ with $\pi(g)=\chi(g)\1$ for all $g\in G$. 
	Consequently, for all $g\in G$ and $\varphi\in L^2(G/K)$ we have $\varphi(ga)=\chi(g)\varphi(a)$ for almost every point $a\in G/K$. 
	
	In particular, 
	for every $\varphi\in L^2(G/K)$ there exists $c_\varphi\in\RR$ with 
	$\vert \varphi (a)\vert=c_\varphi$ for almost every $a\in G/K$.  
	If there exist  $a_0,a_1\in G/K$  with $a_0\ne a_1$, then, 
	for $j=0,1$,  
	we choose a compact neighbourhood  $V_j\subseteq G/K$ of $a_j$ with $V_0\cap V_1=\emptyset$. 
	Then, since $G/K$ is a locally compact Hausdorff space, it follows by an application of Urysohn's Lemma as e.g. in \cite[2.12]{Ru87} that there exists a continuous function with compact support $\varphi\colon G/K\to \RR$ with $\varphi\vert_{V_1}\equiv1$ 
	and $\varphi\vert_{V_0}\equiv0$ for $j=0,1$. 
	Since the measure of $G/K$ is  $G$-invariant, its support is equal to $G/K$, 
	and then the measure of each of the sets $V_0$ and $V_1$ is positive, since both these sets have nonempty interior. 
	Thus, there is no $c_\varphi\in\RR$ with 
	$\vert \varphi (a)\vert=c_\varphi$ for almost every $a\in G/K$. 
	This contradiction shows that 
	the space $G/K$  does not contain two different points. 
	Thus $G=K$, and this completes the proof. 
\end{proof}

Lemmas \ref{nonfact} and \ref{modelind} are needed in the proof of Lemma~\ref{co-L5}. 

\begin{lemma}
	\label{nonfact}
	Let $K\times X\to X$, $(k,x)\mapsto kx$, be a continuous action of a compact group on  a locally compact space. 
	For a $K$-invariant positive Radon measure $\mu\ne 0$ on $X$, consider the corresponding Hilbert space $\Hc:=L^2(X,\mu)$ and unitary representation 
	$\pi\colon K\to U(\Hc)$, $(\pi(k)f)(x):=f(k^{-1}x)$ for $x\in X$, $k\in K$, and $f\in \Hc$. 
	If $\pi$ is a factor representation, then for every $k\in K$ and $x\in\supp\mu$ we have $kx=x$. 
\end{lemma}

\begin{proof}
We start with a general remark. 
If $K$ is a locally compact group of type I and $\pi \colon K\to U(\Hc_\pi)$ is a factor representation such that $\pi$ contains a 1-dimensional sub-representation,  
then there is $\chi\in \Hom(K,\TT)$ such that $\pi(k)=\chi(k) \1$. 
(See \cite[5.3.3--5 and 5.4.11]{Di64}.)

Assume now that in addition $K$ is a compact group and $\pi$ is as in the statement. 
	Since $\mu$ is a positive Radon measure and $\mu\ne0$, 
	we can select a compact subset $C_0\subseteq X$ with $\mu(C_0)>0$. 
	Since the group $K$ is compact, it follows that the set $C:=KC_0$ is compact as well. 
	Moreover, the set $C$ is $K$-invariant and $C_0\subseteq C$, hence $\mu(C)>0$, 
Denote by $\Hc_{0}$ the 1-dimensional vector space consisting of constant functions 
on~$C$. 
	Since the set $C$ is compact and $\mu$ is a Radon measure, 
	$\Hc_{0}$  can be seen as closed subspace of $\Hc$ in a natural way. Moreover 
	we have $\pi(K)\Hc_{0}=\Hc_{0}$.  
	Hence the representations $\pi_0\colon K \to U(\Hc_{0})\simeq \TT$, 
	$\pi_0(k)= \pi(k)\vert_{\Hc_0}$ is an 1-dimensional sub-representation of $\pi$. 
	
	By the discussion above there exists $\chi\in\Hom(K,\TT)$ with 
	$\pi(k)=\chi(k)\1$ for every $k\in K$. 
	Taking into account the definition of the representation $\pi$, 
	we obtain in particular $\vert f(k^{-1}x)\vert=\vert f(x)\vert$ for arbitrary $k\in K$, $f\in \Hc$ and $\mu$-a.e. $x\in X$. 
	Using Urysohn's lemma as in the proof of Lemma~\ref{L9}, one can then show that 
	for every  $x\in \supp\mu$ we have $Kx=\{x\}$. 
\end{proof}

\begin{lemma}
	\label{modelind}
	Let $\Vc$ be a finite-dimensional real vector space, $K$ a compact group with a continuous homomorphism $\alpha\colon K\to\GL(\Vc)$, and consider $G:=\Vc\rtimes_\alpha K$. 
	For $\xi \in \Vc^*$ set 
	$K_\xi:=\{k\in K\mid\xi\circ\alpha(k)=\xi\}$. 
	We then define $G_\xi:=\Vc\rtimes_\alpha K_\xi$.
	For  $\chi:=\ee^{\ie\langle\xi,\cdot\rangle}\in\Hom(\Vc,\TT)$, let $\widetilde{\chi}\colon G_\xi\to\TT$ be given by $\widetilde{\chi}(v,k):=\chi(v)$. 
	
	Then $\widetilde{\chi}\in\Hom(G_\xi,\TT)$ and the following assertions hold: 
	\begin{enumerate}[{\rm(i)}]
		\item
		\label{modelind_item1} 
		The representation 
		\begin{equation*}
			\pi\colon G\to U(L^2(K/K_\xi)), \quad (\pi(v,k)\psi)(\ell K_\xi):=\chi(\alpha(\ell^{-1}v))\psi(k^{-1}\ell K_\xi)
		\end{equation*}
		is unitary equivalent to the representation $\Ind_{G_\xi}^G(\widetilde{\chi})$. 
		\item\label{modelind_item2} 
		If we define $\tau\colon K/K_\xi\to\Vc^*$, $\tau(\ell K_\xi):=\xi\circ\alpha(\ell^{-1})$ and $\mu_\xi$ is the probability Radon measure on $\Vc^*$ obtained as the pushforward through $\tau$ of the $K$-invariant probability Radon measure of $K/K_\xi$, then 
		\begin{equation}
					\label{modelind_proof_eq1}
					\supp\mu_\xi=K.\xi:=\{\xi\circ\alpha(\ell)\mid \ell\in K\}\subseteq\Vc^*.	
			\end{equation}
		and 
		the representation 
		$$\widetilde{\pi}\colon G\to U(L^2(\Vc^*,\mu_\xi)),\quad 
		(\widetilde{\pi}(v,k)f)(\eta):=\ee^{\ie\langle\eta,v\rangle}f(\eta\circ\alpha(k))$$
		is also unitary equivalent to the representation $\Ind_{G_\xi}^G(\widetilde{\chi})$. 
		\item\label{modelind_item3} 
		The representations $\widetilde{\pi}$, $\pi$, and $\Ind_{G_\xi}^G(\widetilde{\chi})$ 
		are irreducible. 
		\item\label{modelind_item4} 
		For arbitrary $t\in\RR$ we define $\xi_t:=\ee^{-t}\xi\in\Vc^*$, 
		$\chi_t\colon G_\xi\to\TT$, $\chi_t(k,v):=\ee^{\ie\langle\xi_t,v\rangle}$.
		Then $\chi_t\in\Hom(G_\xi,\TT)$ and $[\Ind_{G_\xi}^G(\chi_t)]\in\widehat{G}$. 
		Moreover, for every $[\sigma]\in\widehat{G}$ with the property that 
		$[\Ind_{G_\xi}^G(\chi_t)]$ converges to $[\sigma]$  in $\widehat{G}$ for $t\to\infty$, 
		we have $\Vc\subseteq\Ker\sigma$. 
	\end{enumerate}
\end{lemma}

\begin{proof}
	The property $\widetilde{\chi}\in\Hom(G_\xi,\TT)$ is straightforward. 
	
	\eqref{modelind_item1} 
	Since the group $K$ is compact, both groups $G$ and $G_\xi$ are unimodular, hence Remark~\ref{unimod} applies. 
	Let us define $\Hc^0$ as the space of $\varphi\in\Cc(G)$ such that
	\begin{eqnarray*}
	 & \varphi(gh)  =\widetilde{\chi}^{-1}(h)\varphi(g) \quad (\forall h\in G_\xi,g\in G) \\
	& g G_\xi \mapsto \vert \varphi(g)\vert \text{ has compact support in }  G/G_\xi.
	\end{eqnarray*} 
Then $\Hc^0$ is dense in the representation Hilbert space $\Hc$ of the induced representation $\pi_0:=\Ind_{G_\xi}^G(\widetilde{\chi})$, 
	and $(\pi_0(g)\varphi)(x)=\varphi(g^{-1}x)$ for all $g,x\in G$ and $\varphi\in\Hc^0$. 
	
	For every $\varphi\in\Hc^0$, $v\in\Vc$, $k\in K$, and $k_0\in K_\xi$ we have $\varphi(0,kk_0)=\varphi(0,k)$ and 
	$$\varphi(v,k)=\varphi((0,k)\cdot(\alpha(k^{-1})v,\1))
	=\chi(\alpha(k^{-1})v)^{-1}\varphi(0,k).$$
	Then the  operator 
	$$U\colon \Hc\to L^2(K/K_\xi),\quad (U\varphi)(kK_\xi):=\varphi(0,k)$$
	is unitary,
	with its inverse 
	$$U^{-1}\colon L^2(K/K_\xi)\to\Hc,\quad 
	(U^{-1}\psi)(v,k)=\chi(\alpha(k^{-1})v)^{-1}\psi(kK_\xi).
	$$
	We now check that $U$ intertwines 
	the representations $\pi$ and~$\pi_0$. 
	Indeed, for arbitrary $v\in\Vc$, $k,\ell\in K$, and $\psi\in\Cc(K/K_\xi)$, 
	we have 
	\begin{align*}
		(U\pi_0(v,k)U^{-1}\psi)(\ell K_\xi)
		&=(\pi_0(v,k)U^{-1}\psi)(0,\ell) \\
		&=(U^{-1}\psi)((v,k)^{-1}\cdot (0,\ell))\\ 
		&=\chi(-\alpha(\ell^{-1})v)^{-1}\psi(k^{-1}\ell K_\xi) \\
		&=(\pi(v,k)\psi)(\ell K_\xi).
	\end{align*}
	Since $\Cc(K/K_\xi)$ is dense in $L^2(K/K_\xi)$, this completes the proof of Assertion~\eqref{modelind_item1}. 
	
	\eqref{modelind_item2} 
	By  Assertion~\eqref{modelind_item1}, it suffices to show that 
	the representations $\widetilde{\pi}$ and $\pi$ are unitary equivalent. 
	To this end we first note that for every $k\in K_\xi$ and $\ell\in K$ we have 
	$\xi\circ\alpha(\ell^{-1})=\xi\circ\alpha((\ell k)^{-1})$ 
	hence the mapping $\tau\colon K/K_\xi\to\Vc^*$ is well defined  
	and a homeomorphism onto its image. 
	
	A simple computation shows that 
	the support of the measure $\mu_\xi$ is the compact $K$-orbit given by \eqref{modelind_proof_eq1}.  
	Hence the operator 
	$$W\colon L^2(K/K_\xi)\to L^2(\Vc^*,\mu),\quad 
	(W\psi)(\xi\circ\alpha(\ell^{-1})):=\psi(\ell K_\xi)$$
	is well defined and unitary, 
	with its inverse 
	$$W^{-1}\colon L^2(\Vc^*,\mu_\xi)\to L^2(K/K_\xi),\quad W^{-1}f:=f\circ\tau.$$
	Then for arbitrary $v\in\Vc$, $k,\ell\in K$, and $f\in L^2(\Vc^*,\mu)$ we have 
	\begin{align*}
		(W\pi(v,k)W^{-1}f)(\xi\circ\alpha(\ell^{-1}))
		&=(\pi(v,k)W^{-1}f)(\ell K_\xi) \\
		&=\ee^{\ie\langle\xi,\alpha(\ell^{-1})v\rangle}(W^{-1}f)(k^{-1}\ell K_\xi) \\
		&=\ee^{\ie\langle\xi\circ\alpha(\ell^{-1}),v\rangle}f(\xi\circ\alpha(\ell^{-1}k)) \\
		&=(\widetilde{\pi}(v,k)f)(\xi\circ\alpha(\ell^{-1}))
	\end{align*}
	which completes the proof of Assertion~\eqref{modelind_item2}. 
	
	\eqref{modelind_item3} 
	The representations $\widetilde{\pi}$, $\pi$, and $\Ind_{G_\xi}^G(\widetilde{\chi})$ are unitary equivalent by Assertions~\eqref{modelind_item1}--\eqref{modelind_item2} 
	hence it suffices to prove that one of them is irreducible. 
	
	An application of Mackey's theorem 
	\cite[Thm. 8.1]{Ma58} shows that the representation $\Ind_{G_\xi}^G(\widetilde{\chi})$ is irreducible. 
	Alternatively, one can directly prove that $\widetilde{\pi}$ is irreducible, as follows.
	We first note the equality 
	\begin{equation}
		\label{modelind_proof_eq2}
		\overline{\spa\{\ee^{\ie\langle\cdot,v\rangle}\mid v\in\Vc\}}^{w^*}=L^\infty(\Vc^*,\mu_\xi).
	\end{equation}
	Indeed, in order to check that the set $\{\ee^{\ie\langle\cdot,v\rangle}\mid v\in\Vc\}$ 
	spans a $w^*$-dense linear subspace of $L^\infty(\Vc^*,\mu_\xi)$, 
	we recall the Banach space isomorphism 
	$L^\infty(\Vc^*,\mu_\xi)\simeq (L^1(\Vc^*,\mu_\xi))^*$. 
	Let $f\in L^1(\Vc^*,\mu_\xi)$ with
	$\int_{\Vc^*}\ee^{\ie\langle\cdot,v\rangle}f\de\mu_\xi=0$ for every $v\in\Vc$.  
	Then $f=0$ in $L^1(\Vc^*,\mu_\xi)$ by the injectivity of the Fourier transform. 
	In some more detail, the measure $f\de\mu_\xi$ has compact support, hence it defines a tempered distribution on $\Vc^*$. 
		The Fourier transform of that tempered distribution vanishes on~$\Vc$, 
		hence $f\de\mu_\xi$, which is equivalent to $f=0\in L^1(\Vc^*,\mu_\xi)$.
	
	Returning to proving that the representation $\widetilde{\pi}$ is irreducible, 
	we actually show that if $T\in\Bc(L^2(\Vc^*,\mu_\xi))$ and $T\widetilde{\pi}(v,k)=\widetilde{\pi}(v,k)T$ for every $v\in\Vc$ and $k\in K$, then $T$ is a scalar multiple of the identity operator. 
	In fact, since $T\widetilde{\pi}(v,\1)=\widetilde{\pi}(v,\1)T$ for every $v\in\Vc$, 
	it follows by  \eqref{modelind_proof_eq1} that $T$ commutes with the multiplication operator defined by an arbitrary element in $L^\infty(\Vc^*,\mu_\xi)$. 
	The set of these multiplication operators is a maximal abelian $*$-subalgebra of $\Bc(L^2(\Vc^*,\mu_\xi))$, hence it follows that $T$ is in turn a multiplication operator by a function $\varphi_0\in L^\infty(\Vc^*,\mu_\xi)$. 
	Then the condition $T\widetilde{\pi}(0,k)=\widetilde{\pi}(0,k)T$ for every $k\in K$ 
	implies that the function $\varphi_0$ is constant almost everywhere on $\Vc^*$ with respect to the measure $\mu_\xi$. 
	This further implies that the operator $T$ is a scalar multiple of the identity operator, 
	and we are done. 
	
	\eqref{modelind_item4} 
	We first note that  $K_{\xi}=K_{\xi_t}$
	$G_\xi= G_{\xi_t}$  for every $t\in\RR$.
	Then we define 
	$\tau_t\colon K/K_\xi\to\Vc^*$, $\tau(\ell K_\xi):=\xi_t\circ\alpha(\ell^{-1})$ and $\mu_{\xi_t}$ is the probability Radon measure on~$\Vc^*$ obtained as the pushforward through $\tau_t$ of the $K$-invariant probability Radon measure of $K/K_\xi$. 
	Then, by \eqref{modelind_item2}, 
	applied for $\xi_t\in\Vc^*$ (which satisfies $K_{\xi_t}=K_\xi$) 
	the unitary representation 
	\begin{equation}
		\label{modelind_proof_eq3}
		\widetilde{\pi}_t\colon G\to U(L^2(\Vc^*,\mu_{\xi_t})),\quad 
		(\widetilde{\pi}_t(v,k)f)(\eta):=\ee^{\ie\langle\eta,v\rangle}f(\eta\circ\alpha(k))
	\end{equation}
	is unitary equivalent to the representation $\Ind_{G_\xi}^G(\widetilde{\chi_t})$. 
	Therefore, if $[\Ind_{G_\xi}^G(\chi_t)]$ converges to $[\sigma]$  in $\widehat{G}$ for $t\to\infty$ then, by the continuity of the restriction of unitary representations 
	\cite[Lemma 1.11 and Thm. 2.3]{Fe62}, 
	it follows that $\widetilde{\pi}_t\vert_\Vc$ converges to $\sigma\vert_\Vc$  in the space of equivalence classes of representations of the abelian Lie group $\Vc$ 
	for $t\to\infty$. 
	For every $t\in\RR$ we have by \eqref{modelind_proof_eq1}
	\begin{equation*}
		\supp\mu_{\xi_t}=\ee^{-t}\supp\mu_\xi
	\end{equation*}
	and $\supp\mu_\xi\subseteq\Vc^*$ is a compact. 
	Thus, for $t\to\infty$, $\{0\}$ is the only limit point of 
	$\supp\mu_{\xi_t}$ in the space $\Cl(\Vc^*)$ endowed with its upper topology. 
	(See e.g., \cite{BB-LMS}.) 
	On the other hand, we obtain by \eqref{modelind_proof_eq3} the direct integral decomposition
	$$\widetilde{\pi}_t\vert_\Vc
	=\int^{\oplus}_{\Vc^*}\ee^{\ie\langle\eta,\cdot\rangle}\de\mu_{\xi_t}(\eta)$$
	hence $\widetilde{\pi}_t\vert_\Vc$ is weakly equivalent to the family $\{\ee^{\ie\langle\eta,\cdot\rangle}\mid \eta\in\supp\mu_{\xi_t}\}$ 
	by \cite[Thm. 1.4]{Fe60} or \cite[Thm. 3.1]{Fe62}. 
	It then follows by the definition of the inner hull-kernel topology \cite[\S 2]{Fe62} on the space of nondegenerate representations of $C^*(\Vc)\simeq\Cc_0(\Vc^*)$ 
	that, for $t\to 0$,  the trivial representation of $\Vc$ is the only limit point of $\widetilde{\pi}_t\vert_\Vc$ with respect to the inner hull-kernel topology. 
	Therefore $\sigma\vert_\Vc$ is the trivial representation of $\Vc$, that is, $\Vc\subseteq\Ker\sigma$. 
\end{proof}

\begin{remark}
	\normalfont 
	The method of proof of Lemma~\ref{modelind}\eqref{modelind_item3}  goes back to \cite{BB16}. 
\end{remark}

\section{On the topology of unitary duals of solvable Lie groups}
\label{sect:prelsolv}

This section collects results from representation theory of general solvable Lie groups that are needed in the proof of Theorem~\ref{dual_main} via Proposition~\ref{P13}.
Some of these results go back to \cite{Pu71} and \cite{Pu73}, 
but we use slightly different notation for which we refer to~\cite{BB25}. 

In the main part of the proof of Theorem~\ref{dual_main}, the focus is on Lie groups of type~\I. 
(See Section~\ref{sect:proof}.)
The reason for that is an application  for group $C^*$-algebras of 
the fact that a $C^*$-algebra $\Ac$ is liminary (in particular type \I) if and only if its dual space~$\widehat{\Ac}$ is a topological space of type $T_1$, that is, every finite subset is closed, cf. \cite[9.5.3]{Di64}. 
Nevertheless, 
 we also record here a very few facts on solvable Lie groups without any restriction on their type, since these results are not getting much simpler for 
 groups of type~\I, 
 and, moreover, they provide a proper perspective on the basic techniques that we need from representation theory of solvable Lie groups.
 
Unless otherwise specified, $G$ is a 1-connected solvable Lie group with its Lie algebra~$\gg$. 
 We denote by $D\subseteq G$ the closed connected subgroup corresponding 
 to the derived Lie algebra $\dg:=[\gg,\gg]$. 
 (Compare Remark~\ref{one-dim}.)
 
 We now summarize basic notation  related to the Puk\'anszky correspondence, 
 cf. \cite[\S 2]{BB25} and the references therein. 
 For arbitrary  $\xi\in\gg^*$ we define:
 \begin{itemize}
 	\item 	$G(\xi):=\{g\in G\mid g\xi=\xi\}$ is the stabilizer at $\xi$ of the coadjoint action and
 	 the connected component of $\1\in G(\xi)$ is denoted $G(\xi)_\1$.
 	The Lie algebra of both $G(\xi)$ and $G(\xi)_\1$ is 
 	$\gg(\xi):=\{x\in\gg\mid(\forall y\in\gg)\  \langle\xi,[x,y]\rangle=0\}$. 
 	\item  $\chi_\xi\in\Hom(G(\xi)_\1,\TT)$ with
 	$\chi_\xi(\exp_G x)=\ee^{\ie \langle\xi,x\rangle}$. 
 	\item 
 	$\overline{G}(\xi):=\{g\in G(\xi)\mid (\forall h\in G(\xi))\ ghg^{-1}h^{-1}\in\Ker\chi_\xi\}$ is the reduced stabilizer at $\xi$. 
 	\item $\mathop{G}\limits^{\mathsout{\wedge}}(\xi):=
 	\{\chi\in\Hom(\overline{G}(\xi),\TT)\mid \chi\vert_{G(\xi)_\1}=\chi_\xi \}.$
 \end{itemize}
There is a group isomorphism  
\begin{equation}
\label{free}
G(\xi)/G(\xi)_\1\simeq\ZZ^r
\end{equation}
 for a suitable integer $r\ge0$, 
cf. \cite[Eq. (2.3)]{BB25} and the references therein.

The group $G$ acts on the set $\Bun(\gg^*):=\bigsqcup\limits_{\xi\in\gg^*}\{\xi\}\times \mathop{G}\limits^{\mathsout{\wedge}}(\xi) 
$, by $(g, (\xi, \chi))\mapsto (g\xi, g\chi)$, where $g\chi\in \mathop{G}\limits^{\mathsout{\wedge}}(g\xi)$ 
is defined by  $(g\chi)(h)=\chi(g^{-1}hg)$ for all $h\in \overline{G}(g\xi)$.  
The map  $\tau\colon \Bun(\gg^*)\to\gg^*$,  $\tau(\xi,\chi):=\xi$, 
 intertwines 
 the coadjoint action $G\times\gg^*\to\gg^*$ 
 and the action $G\times \Bun(\gg^*)\to\Bun(\gg^*)$, $(g,(\xi,\chi))\mapsto (g\xi,g\chi)$.

Let \Rel\  be 
the equivalence relation on $\gg^*$ defined  by 
$$\xi \text{ is \Rel-equivalent to $\eta$  } \iff \overline{G\xi}=\overline{G\eta}, $$
where  $\xi,\eta\in\gg^*.$
The equivalence classes of \Rel\ are called \emph{quasi-orbits} of the coadjoint action and their set is denoted by $(\gg^*/G)^\sim$. 
 
For every coadjoint quasi-orbit $\Oc\in(\gg^*/G)^\sim$ 
the subset $\Bun(\Oc):=\tau^{-1}(\Oc)\subseteq\Bun(\gg^*)$ is $G$-invariant 
and has the natural structure of a smooth manifold.  
The closures (in $\Bun(\Oc)$) of $G$-orbits  constitute a partition of $\Bun(\Oc)$. 
We denote by
	$(\Bun(\Oc)/G)^\approx=\{\overline{Gp}\mid p=(\xi,\chi)\in\Bun(\Oc)\}$
the set of these orbit closures.
Thus
$$\RelS:=\bigcup\limits_{\Oc\in (\gg^*/G)^\sim}(\Bun(\Oc)/G)^\approx,$$ 
is a  partition of $\Bun(\gg^*)$.
Each $\O\in\RelS$ is called a \emph{generalized orbit of the coadjoint action of~$G$}. 
There are bijective mappings  
\begin{equation}
\label{kerl0}
\xymatrix{ 
	\RelS \ar[r]^{\ell\quad} & \stackrel{\frown}{G}_\nor \ar[r]^{\ker\quad} & \Prim(G)
}
\end{equation}
where $\ell$ is called the \emph{Puk\'anszky correspondence}. 
Here we denote by $\stackrel{\frown}{G}$ the set of all quasi-equivalence classes $[\pi]^\frown$ of factor unitary representations $\pi$ of the Lie group $G$, 
and  $\stackrel{\frown}{G}_\nor\subseteq \stackrel{\frown}{G}$ is the subset corresponding to normal representations, that is, unitary representations $\pi\colon G\to\Bc(\Hc_\pi)$ for which the von Neumann algebra $\pi(G)''\subseteq \Bc(\Hc_\pi)$ is a factor with a semifinite normal faithful trace $\tau$ satisfying $0<\tau(\pi(a^*a))<\infty$ for some element $a\in C^*(G)$. 
We omit the construction of $\ell$, but a description of the composition $\ker\circ\ell$ is given in \eqref{kerl} below.

\subsection{Remarks on metabelian Lie groups}
In this subsection we derive special properties of the above structures in the special case when the Lie group $G$ is metabelian, i.e., its derived group $D$ is abelian. 

\begin{remark}\label{indreps}
	\normalfont
	Let $\xi\in\gg^*\setminus\dg^\perp$  and $\chi\in\mathop{G}\limits^{\mathsout{\wedge}}(\xi)$, hence $p:=(\xi,\chi)\in\Bun(\gg^*)$. 
	There exists a complex polarization $\hg\subseteq \dg_\CC $ of  $\xi\vert_\dg$ 
	which is invariant to the reduced stabilizer $\overline{G}(\xi)$, 
	cf.  \cite[pages 493--494]{Pu71}. 
	Define $\mg := \dg \cap \hg$ and let $M:=\exp \mg \subset D$ be the subgroup corresponding to $M$. 
	Then $K_0:= \overline{G}(\xi)M$ is a closed subgroup of $G$ and there exists a 
	unique group homomorphism 
	\begin{equation}
	\label{indreps_eq1}
	\varphi_{\xi,\chi}\colon K_0 \to \TT, 
	\end{equation}
	to be denoted $\varphi_p$ (where $p:=(\xi,\chi)$ as above) or even just $\varphi$ for simplicity, 
	such that $\varphi\vert_{\overline{G}(\xi)}= \chi$ and $\varphi\vert_{M}=\ee^{\ie \xi}$, 
	where $\ee^{\ie \xi}\vert_M\colon M\to\TT$ is the group homomorphism whose differential at the unit element is the Lie algebra homomorphism $\ie\xi\vert_\mg\colon\mg=\dg\cap\hg\to\ie \RR$. 
	
	We define $K:=  \overline{G}(\xi)D=D\overline{G}(\xi)$. 
	This is again a closed subgroup of $G$ 
	and $K_0\subseteq K$
	(see \cite[Prop. 6.1, page 503]{Pu71}).
	We obtain by holomorphic induction from $K_0$ to $K$ the irreducible unitary  representation 
		$\ind(\hg,p;K)$ 
	of $K$, which is a sub-representation of the induced representation 
	$\Ind_{K_0}^K \varphi$. 
Using	the ordinary unitary induction 
	recalled at the beginning of Subsection~\ref{subsect:ind}, 
we define 
	$$T(p):=\Ind_K^G(\ind(\hg,p;K))$$
	and this is a semifinite factor representation of $G$, 
	cf. the proof of \cite[Cor. 7.1, pages 509]{Pu71} and \cite[Lemma 7.3, page 510]{Pu71}.
	Moreover, if we select $\Oc\in(\gg^*/G)^\sim$ with $\xi\in\Oc$ 
	and we denote $\O:=\overline{Gp}\subseteq\Bun(\Oc)$
	then, by \cite[Prop. 1 and proof of its Cor., page 93]{Pu73},
	\begin{equation}
	\label{kerl}
	\ker\ell(\O)=\ker T(p)\in\Prim(G)
	\end{equation}
	Recall that we use the same notation  
	both for the unitary representation $T(p)$
	and for its corresponding nondegenerate $*$-representation of $C^*(G))$. 
	
	Under the additional assumption that $G$ is metabelian, that is, $\dg=[\gg,\gg]$ 
	is abelian, then $\hg= \dg_\CC$ and $K_0=K$. 
	Hence 
	$$\ind(\hg,p;K)=\varphi_p\colon \overline{G}(\xi)D=K\to \TT$$ 
	and the above equalities show that, for $p=(\xi,\chi)\in\Bun(\gg^*)$, 
		\begin{equation}
		\label{indreps_eq2}
		T(p)=\Ind_K^G(\varphi_p)
	\text{ and }
	\ker\ell(\O)=\ker (\Ind_K^G(\varphi_p))\in\Prim(G).
	\end{equation}
\end{remark}

\begin{lemma}
\label{L3}
If $G$ is a 1-connected metabelian Lie group, then the following assertions hold: 
\begin{enumerate}[{\rm(i)}]
\item\label{L3_item1} For every $\xi\in\gg^*$  and $\chi\in\mathop{G}\limits^{\mathsout{\wedge}}(\xi)$ 
the unitary representation $\Ind_{D\overline{G}(\xi)}^G(\varphi_{\xi,\chi})$ is a factor representation. 
\item\label{L3_item2} For every $\Pc\in\Prim(G)$ 
there exist $\xi\in\gg^*$  and $\chi\in\mathop{G}\limits^{\mathsout{\wedge}}(\xi)$ 
with 
$$\Pc=\ker(\Ind_{D\overline{G}(\xi)}^G(\varphi_{\xi,\chi})).$$ 
\end{enumerate}
\end{lemma}

\begin{proof}
\eqref{L3_item1}
This follows by  Remark~\ref{indreps} since $D\overline{G}(\xi)=K$ and $(\xi,\chi)=p\in\Bun(\gg^*)$. 

\eqref{L3_item2}
Use \cite[Thm. 1]{Pu73}
along with \eqref{indreps_eq2} in Remark~\ref{indreps}. 
\end{proof}

The following lemma plays an important role in Section~\ref{sect:solvtypeI}, 
namely in the proof of Lemma~\ref{L10}. 

\begin{lemma}
\label{L7}
If $G$ is a 1-connected solvable Lie group and $\xi\in\gg^*$ whose closed subgroup $D\overline{G}(\xi)\subseteq G$ is connected, then $\overline{G}(\xi)$ is connected. 
\end{lemma}

\begin{proof}
As noted in \cite[page 494, line 7]{Pu71} (where the commutator subgroup $D$ is denoted by $L$), the subgroup $D\overline{G}(\xi)\subseteq G$ is closed and 
\begin{equation}
\label{L7_eq0}
D\cap\overline{G}(\xi)=D\cap G(\xi)=\exp(\dg\cap\gg(\xi)).
\end{equation}  
Since this last group is connected, we obtain 
\begin{equation}
\label{L7_eq2}
D\cap\overline{G}(\xi)\subseteq G(\xi)_\1.
\end{equation}
Since $G(\xi)_\1$ is equal to the connected component of $\1\in\overline{G}(\xi)$, 
it follows by \cite[Prop. 9.1.15]{HN12} that $G(\xi)_\1$ is clopen in $\overline{G}(\xi)$, 
hence  the group
\begin{equation}
	\label{L7_eq3}	\overline{G}(\xi)/G(\xi)_\1\text{ is at most countable and its quotient topology is  discrete.}
\end{equation}
See also \eqref{free}. 

On the other hand, by the second isomorphism theorem for Lie groups 
\cite[Ch. III, \S 3, no. 8, Prop. 31]{Bo06}
we have a Lie group isomorphism 
\begin{equation*}
(D\overline{G}(\xi))/D \simeq \overline{G}(\xi)/(D\cap \overline{G}(\xi)).
\end{equation*}
The left-had side of the above relation is a connected Lie group by the hypothesis that 
$D\overline{G}(\xi)\subseteq G$ is connected, and then 
the quotient group $\overline{G}(\xi)/(D\cap \overline{G}(\xi))$ is connected as well. 
Since \eqref{L7_eq2} ensures that the natural map 
$$\overline{G}(\xi)/(D\cap \overline{G}(\xi))\to \overline{G}(\xi)/G(\xi)_\1$$ 
is well defined, continuous and surjective, 
it then follows that $\overline{G}(\xi)/G(\xi)_\1$ is connected. 
Then, by \eqref{L7_eq3}, we obtain $\overline{G}(\xi)=G(\xi)_\1$, 
that is, 
$\overline{G}(\xi)$ is connected. 
\end{proof}

\subsection{Remarks on solvable Lie groups of type I}
In the special case when the solvable Lie group $G$ is type~\I, 
several aspects of the Puk\'anszky correspondence \eqref{kerl0} are simpler and 
covered by the Auslander-Kostant theory \cite{AuKo71}. 
More specifically, if $G$ is type \I, then there are canonical bijective maps $\stackrel{\frown}{G}_\nor\to\widehat{G}$, cf. \cite[5.5.3 and 6.7.5]{Di64}, and 
$\ker\colon \widehat{G}\to\Prim(G)$. 
In addition, every coadjoint quasi-orbit $\Oc\in(\gg^*/G)^\sim$ is actually a coadjoint orbit, and moreover every $G$-orbit in $\Bun(\Oc)$ is closed, 
cf. \cite[Prop. 5.2]{BB25}. 
This subsection is devoted to establishing a few basic facts in this simplified setting, for later reference. 

\begin{notation}
\label{N0}
\normalfont
For an arbitrary 1-connected solvable Lie group $G$ of type~\I\ we denote by 
$$\kappa\colon  \bigsqcup_{\Oc \in \gg^*/G} \Bun(\Oc)/G \to \widehat{G}$$
its corresponding Auslander-Kostant-Puk\'anszky bijective correspondence, 
related to \eqref{kerl0} by the above remarks.   
For every coadjoint orbit $\Oc \in\gg^*/G$  we have the following: 
\begin{enumerate}[{\rm(i)}]
	\item\label{N0_item1} There exists an integer $\rk(\Oc)\ge0$ with $\widehat{\pi_1(\Oc)}\simeq\ZZ^{\rk(\Oc)}$  and a $G$-equivariant bijection $\psi\colon \Oc\times \TT^{\rk(\Oc)}\to\Bun(\Oc)$, 
	where  $\pi_1(\Oc)$ is the fundamental group of $\Oc$, 
	cf. \cite[Eq. (2.6)]{BB25}. 
	For $\xi\in\Oc$ and $r$ as in \eqref{free} we have $\rk(\Oc)=r$. 
	\item\label{N0_item2} The mapping 
	$$\kappa_\Oc:=\kappa\vert_{\Bun(\Oc)/G}\colon 
	\Bun(\Oc)/G\to \widehat{G}$$ 
	is continuous, by 
	\cite[Lemma 8, p. 93 and Thm. 1, p. 114]{Pu73}. 
	\item\label{N0_item3} We have a homeomorphism $\psi_\Oc\colon\TT^{\rk(\Oc)}\to \Bun(\Oc)/G$ by \eqref{N0_item1} above. 
	See also \cite[Prop. 5.2 and its proof]{BB25}.
\end{enumerate}
\end{notation}

The following lemma develops a deformation technique that is based on \eqref{free} above and is needed in  
the proof of Lemma~\ref{L10} in Section~\ref{sect:solvtypeI} . 

\begin{lemma}
	\label{L8}
	Let $G$ be a 1-connected solvable Lie group of type~\I, $\xi\in\gg^*$, and $\xi\in\Hom(G(\xi),\TT)$ with $\chi\vert_{G(\xi)_\1}=\chi_\xi$, that is, $(\xi,\chi)\in\Bun(\gg^*)$. 
	Select $a_1,\dots,a_r\in G(\xi)$ such that the family $a_1G(\xi)_\1,\dots,a_rG(\xi)_\1$ is a basis in the  finitely-generated free abelian group $G(\xi)/G(\xi)_\1$. 
	For $j=1,\dots,r$ denote $z_j:=\chi(a_j)\in\TT$. 
	
	For every $t\in\RR^\times$ 
	there exists 
	 $\chi_t\in \Hom(G(\xi),\TT)$ with $\chi_t\vert_{G(\xi)_\1}=\chi_{t\xi}$ and $\chi_t(a_j)=z_j$ for $j=1,\dots,r$. 
	Moreover, the following assertions hold: 
	\begin{enumerate}[{\rm(i)}]
		\item\label{L8_item1} There exists a unique function 
		$\mathop{\chi}\limits^{\circ}\in  \Hom(G(\xi),\TT)$ with $\mathop{\chi}\limits^{\circ}\vert_{G(\xi)_\1}= 1$ and $\mathop{\chi}\limits^{\circ}(a_j)=z_j$ for $j=1,\dots,r$. 
		\item\label{L8_item2} There exists a unique function 
			$\varphi_{0,\mathop{\chi}\limits^{\circ}}\in\Hom(DG(\xi),\TT)$ with
			$\varphi_{0,\mathop{\chi}\limits^{\circ}}\vert_{G(\xi)}=\mathop{\chi}\limits^{\circ}$ and $\varphi_{0,\mathop{\chi}\limits^{\circ}}\vert_D= 1$.  
		\item\label{L8_item3} We have $\lim\limits_{t\to0}\varphi_{t\xi,\chi_t}=\varphi_{0,\mathop{\chi}\limits^{\circ}}$  uniformly on every compact subset of $DG(\xi)$. 
	\end{enumerate}
\end{lemma}

\begin{proof}
	We first note that for every $t\in\RR^\times$ we have $G(t\xi)=G(\xi)$ hence also $G(t\xi)_\1=G(\xi)_\1$. 
	Since the solvable Lie group $G$ is of type~\I, the coadjoint orbit of $t\xi\in\gg^*$ is integral by \cite[\S V.1]{AuKo71}, that is, there exists $\chi_t^0\in\Hom(G(\xi),\TT)$ 
	with $\chi_t^0\vert_{G(\xi)_1}=\chi_{t\xi}$. 
Moreover, since $a_1G(\xi)_\1,\dots,a_rG(\xi)_\1$ is a basis in the  finitely-generated free abelian group $G(\xi)/G(\xi)_\1$, 
there exists 
$\psi\in  
\Hom(G(\xi),\TT)$ 
with $\psi\vert_{G(\xi)_\1}= 1$ and $\psi(a_j)=z_j\chi_t^0(a_j)^{-1}$. 
We now define 
$$\chi_t\colon G(\xi)\to\TT,\quad \chi_t(a):=\psi(a)\chi_t^0(a).$$
This shows existence and uniqueness of $\chi_t\in \Hom(G(t\xi),\TT)$ as in the statement. 
(See also \cite[page 492]{Pu71}.) 
We now turn to the remaining assertions. 

	\eqref{L8_item1}
The function $\mathop{\chi}\limits^{\circ}\in  \Hom(G(\xi),\TT)$ can be constructed just as $\psi$ above. 
	
	\eqref{L8_item2}
	By  
	\eqref{L7_eq0} 
	we have 
	$D\cap G(\xi)=\exp(\dg\cap\gg(\xi))\subseteq G(\xi)_\1$ 
	hence $\mathop{\chi}\limits^{\circ}\vert_{D\cap G(\xi)}= 1$. 
	This shows that 
	$$\mathop{\chi}\limits^{\circ}\in\Hom(G(\xi)/(D\cap G(\xi)),\TT)\hookrightarrow \Hom(G(\xi),\TT).$$
	Using the group isomorphism $G(\xi)/(D\cap G(\xi))\simeq (DG(\xi))/D$, 
	we then obtain 
	$$\varphi_{0,\mathop{\chi}\limits^{\circ}}\in\Hom((DG(\xi))/D,\TT)\hookrightarrow\Hom(DG(\xi),\TT)$$
	as required. 
	
	\eqref{L8_item3} This is straightforward. 
\end{proof}

\section{Rigidity within the class of solvable Lie groups}
\label{sect:solvtypeI}

This section performs the second of the main steps towards proving Theorem~\ref{dual_main}, as explained in the Introduction. 
Specifically, we prove the following fact, which is actually  Theorem~\ref{dual_main} restricted to the class of 1-connected solvable Lie groups. 

\begin{proposition}
	\label{P13}
	If $H$ is a 1-connected solvable Lie group of type~\I\ 
	and there exists an integer $n\ge 1$ such that for the $(2n+1)$-dimensional Heisenberg group $H_n$ we have a homeomorphism  $\widehat{H}\simeq\widehat{H_n}$  then there is a Lie group isomorphism $H\simeq H_n$. 
\end{proposition}

 We first note here that, since the Heisenberg group is a liminary group, it follows by\cite[9.5.3]{Di64} that $H$ is in turn liminary, in particular type \I. 
 Therefore we turn our attention to  type \I\ solvable Lie groups.

The proof of the above Proposition~\ref{P13}  will be given towards the end of this section, 
since it requires several lemmas.

\begin{lemma}
\label{L5}
Let $G$ be a 1-connected solvable Lie group of type~\I\  
with a mapping $\iota\colon\widehat{G}\to X$ that is a homomeomorphism onto 
a closed subset of a topological space $X$ with a Heisenberg partition  $X=\Gamma_1\sqcup\Gamma_2$. 
Then for every coadjoint orbit $\Oc\in\gg^*/G$ 
there exists $j\in\{1,2\}$ 
with  $\kappa(\Bun(\Oc)/G)\subseteq \Gamma_j\cap\iota(\widehat{G})$. 
If $\rk(\Oc)\ge 1$, then actually $\kappa(\Bun(\Oc)/G)\subseteq \Gamma_2\cap\iota(\widehat{G})$. 
\end{lemma}

\begin{proof}
	We may assume without loss of generality that $\widehat{G}$ is a closed subset of $X$ and $\iota\colon\widehat{G}\hookrightarrow X$ is the corresponding set inclusion. 
	
If $\rk(\Oc)=0$, then the assertion is immediate. 

If $\rk(\Oc)\ge 1$, using Lemma~\ref{L4} 
for restrictions of the continuous injective map $$\kappa\circ\psi_\Oc=\kappa_\Oc\circ\psi_\Oc\colon \TT^{\rk(\Oc)}\to\widehat{G}$$ 
to circles contained in the torus $\TT^{\rk(\Oc)}$, 
we obtain $\kappa \circ \psi_\Oc(\TT^{\rk(\Oc)}) \subseteq \Gamma_2$. 
Since $\psi_\Oc$ is a homeomorphism by Notation~\ref{N0}\eqref{N0_item3}, 
the assertion follows. 
\end{proof} 

\begin{lemma}
	\label{L6}
	Let $G$ be a 1-connected metabelian Lie group of type~\I\ 
	with a mapping $\iota\colon\widehat{G}\to X$ that is a homomeomorphism onto
	a closed subset of a topological space $X$ with a Heisenberg partition  $X=\Gamma_1\sqcup\Gamma_2$.  
	If $(\xi,\chi)\in\Bun(\gg^*)$ and  
	$\ker(\Ind_{DG(\xi)}^G(\varphi_{t\xi,\chi_t}))\in\iota^{-1}(\Gamma_1)$, 
	then  $G(\xi)=G(\xi)_\1$. 
\end{lemma}

\begin{proof}
As in the proof of Lemma~\ref{L5}, 
we  assume that $\widehat{G}$ is a closed subset of $X$ and $\iota\colon\widehat{G}\hookrightarrow X$ is the corresponding set inclusion. 
	
The hypothesis makes sense since the unitary representation  $\Ind_{DG(\xi)}^G(\varphi_{t\xi,\chi_t})$ is a factor representation by Lemma~\ref{L3}\eqref{L3_item1}, hence 
$$\ker(\Ind_{DG(\xi)}^G(\varphi_{t\xi,\chi_t}))
\in\Prim(G)\simeq\widehat{G}\subseteq 
X=\Gamma_1\sqcup\Gamma_2.$$
If we denote by $\Oc\in\gg^*/G$ the coadjoint orbit of $\xi$, then by the hypothesis along with Lemma~\ref{L5} we obtain $\rk(\Oc)=0$, that is, $G(\xi)=G(\xi)_\1$. 
\end{proof}

\begin{lemma}
	\label{L10}
Let $G$ be a 1-connected metabelian Lie group of type~\I\ 
with a mapping $\iota\colon\widehat{G}\to X$ that is a homomeomorphism onto
a closed subset of a topological space~$X$ with a Heisenberg partition  $X=\Gamma_1\sqcup\Gamma_2$.  
If $\Oc\in\gg^*/G$ 
and 
$\iota(\kappa(\Bun(\Oc)/G))\subseteq \Gamma_2$, 
then for every $\xi\in\Oc$ we have $G(\xi)=G(\xi)_\1$. 
\end{lemma}

\begin{proof}
	Since the group $G$ is type \I, the mapping $\ker\colon \widehat{G}\to\Prim(G)$ is a homeomorphism. 
	Thus, $\Prim(G)$ is homeomorphic to a closed subset of a topological space $X$. 
	As in the proofs of Lemmas \ref{L5} and \ref{L6}, 
	we  assume that $\widehat{G}$ is actually a closed subset of $X$ and $\iota\colon\widehat{G}\hookrightarrow X$ is the corresponding set inclusion. 
	Then the hypothesis on $\Oc$ takes on the form 
	$\kappa(\Bun(\Oc)/G)\subseteq \Gamma_2\cap\widehat{G}$.

We now fix $(\xi,\chi)\in\Bun(\Oc)$ and, assuming $G(\xi)_\1\subsetneqq G(\xi)$, we will arrive at a contradiction. 

To this end, for every $t\in\RR^\times$, we note that $G(t\xi)=G(\xi)$ 
and we define $(t\xi,\chi_t)\in\Bun(\gg^*)$ as in Lemma~\ref{L8}. 
Then, by Lemma~\ref{L3}\eqref{L3_item1}, the unitary representation 
$$\pi_t:=\Ind_{DG(\xi)}^G(\varphi_{t\xi,\chi_t})$$
is a factor representation. 
By assumption, for every $t\in\RR^\times$, we have $G(t\xi)_\1\subsetneqq G(t\xi)$, 
hence by Lemma~\ref{L6} we have $\ker\, \pi_t\in\Gamma_2\cap\Prim(G)$. 

By Lemma~\ref{L8}\eqref{L8_item3} we have $\lim\limits_{t\to0}\varphi_{t\xi,\chi_t}=\varphi_{0,\mathop{\chi}\limits^{\circ}}$  uniformly on every compact subset of $DG(\xi)$. 
Therefore, by continuity of induction in the sense of \cite[Th. 4.2]{Fe62}, 
it follows that for every sequence $(t_n)_{n\in\NN}$ in $\RR^\times$ with $\lim\limits_{n\to\infty}t_n=0$, 
the unitary representation $\Ind_{DG(\xi)}^G(\varphi_{0,\mathop{\chi}\limits^{\circ}})$ is weakly contained in 
the family of factor representations $\{\pi_{t_n}\mid n\in\NN\}$. 
Using the notation in \eqref{general_eq1}, 
it then directly follows that every point in $\Sp(\Ind_{DG(\xi)}^G(\varphi_{0,\mathop{\chi}\limits^{\circ}}))$ 
is a cluster point of every subnet of $\{[\pi_t]\}_{t\in\RR^\times}$ for $t\to 0$. 
Since $\{\ker\, \pi_t\mid t\in\RR^\times\}\subseteq\Gamma_2\cap\Prim(G)$ 
and $\Gamma_2\cap\Prim(G)$ is a closed subset of $\Prim(G)$ whose relative topology is Hausdorff, 
it follows that the set $\Sp(\Ind_{DG(\xi)}^G(\varphi_{0,\mathop{\chi}\limits^{\circ}}))$  is a singleton, 
that is, there is only one equivalence class of irreducible representations that are weakly contained in $\Ind_{DG(\xi)}^G(\varphi_{0,\mathop{\chi}\limits^{\circ}})$. 
This implies by  \cite[Cor. 5.2.5]{Di64} that $\Ind_{DG(\xi)}^G(\varphi_{0,\mathop{\chi}\limits^{\circ}})$ is a factor representation.  

On the other hand, since 
$\varphi_{0,\mathop{\chi}\limits^{\circ}}\in\Hom(DG(\xi),\TT)$ and $\varphi_{0,\mathop{\chi}\limits^{\circ}}\vert_D=1$, 
it follows that $\varphi_{0,\mathop{\chi}\limits^{\circ}}$ can be regarded as a character of the quotient group $(DG(\xi))/D$, 
which is a closed subgroup of the abelian Lie group $G/D$. 
It then  
follows by \cite[Cor., page 91]{Mo77} that there exists $\widetilde{\varphi}\in\Hom(G,\TT)$ with $\widetilde{\varphi}\vert_{DG(\xi)}=\varphi_{0,\mathop{\chi}\limits^{\circ}}$. 
This further implies 
\begin{align*}
\Ind_{DG(\xi)}^G(\varphi_{0,\mathop{\chi}\limits^{\circ}})
&=\Ind_{DG(\xi)}^G(\widetilde{\varphi}\vert_{DG(\xi)}) \\
&=\Ind_{DG(\xi)}^G(\tau_0\otimes(\widetilde{\varphi}\vert_{DG(\xi)})) \\
&=(\Ind_{DG(\xi)}^G(\tau_0))\otimes \widetilde{\varphi}
\end{align*}
where the last equality follows by \cite[Lemma 4.2, page 260]{Fe62} 
and we have denoted by $\tau_0\in\Hom(DG(\xi),\TT)$ the trivial representation. 
The above equalities show that, if we denote  $\pi_1:=\Ind_{DG(\xi)}^G(\varphi_{0,\mathop{\chi}\limits^{\circ}})$ and 
$\pi_2:=\Ind_{DG(\xi)}^G(\tau_0)$, 
there exists a complex Hilbert space $\Hc$ such that $\pi_1,\pi_2\colon DG(\xi)\to\Uc(\Hc)$ and 
and for every $g\in DG(\xi)$ we have $\pi_1(g)=\widetilde{\varphi}(g)\pi_2(g)$, where $\widetilde{\varphi}(g)\in\TT$. 
This shows that for every operator $T\in\Bc(\Hc)$ and $g\in DG(\xi)$ we have 
$T\pi_1(g)=\pi_1(g)T$ if and only if $T\pi_2(g)=\pi_2(g)T$. 
That is, these representations have the same commutant 
$\pi_1(DG(\xi))'=\pi_2(DG(\xi))'$ hence also the same bicommutant 
$\pi_1(DG(\xi))''=\pi_2(DG(\xi))''$. 
Since we established above that 
$\Ind_{DG(\xi)}^G(\varphi_{0,\mathop{\chi}\limits^{\circ}})$ is a factor representation, it then follows that 
$\Ind_{DG(\xi)}^G(\tau_0)$ is a factor representation. 
Then Lemma~\ref{L9} implies $DG(\xi)=G$. 
In particular, $DG(\xi)$ is connected and then, by Lemma~\ref{L7}, we obtain that $G(\xi)$ is connected, 
which is a contradiction with the assumption $G(\xi)_\1\subsetneqq G(\xi)$. 
This completes the proof. 
\end{proof}

\begin{lemma}
	\label{L11}
	Let $G$ be a 1-connected metabelian Lie group of type~\I\ 
	whose unitary dual $\widehat{G}$ is homeomorphic to 
	a closed subset of a topological space $X$ with a Heisenberg partition. 
	Then for every $\xi\in\gg^*$ we have $G(\xi)=G(\xi)_\1$. 
\end{lemma}

\begin{proof}
		As in the proofs of Lemmas \ref{L5}, \ref{L6} and \ref{L10}, 
	we may assume without loss of generality that $\widehat{G}$ is actually a closed subset of $X$.
	Denote by $X=\Gamma_1\sqcup\Gamma_2$ the Heisenberg partition of~$X$. 
	  
Let $\Oc:=G\xi\in\gg^*/G$. 
It follows by Lemma~\ref{L5} that 
either $\kappa(\Bun(\Oc)/G)\subseteq\Gamma_2\cap\widehat{G}$ 
or $\kappa(\Bun(\Oc)/G)\subseteq\Gamma_1\cap\widehat{G}$ 
and $G(\xi)=G(\xi)_\1$. 
Moreover, if $\kappa(\Bun(\Oc)/G)\subseteq\Gamma_2\cap\widehat{G}$, 
then 
by Lemma~\ref{L10} we obtain again $G(\xi)=G(\xi)_\1$.
\end{proof}

\begin{lemma}
	\label{L12}
	If $H$ is a 1-connected solvable Lie group of type~\I\ whose unitary dual has a Heisenberg partition,  then $H$ is a nilpotent Lie group. 
\end{lemma}

\begin{proof}
We define $H^{[1]}:=[H,H]$, $H^{[2]}:=[H^{[1]},H^{[1]}]$, and $G:=H/H^{[2]}$. 
Then  $G$ is a 1-connected metabelian Lie group of type~\I\ for which $\widehat{G}$ is homeomorphic to a closed subset of the topological space $\widehat{H}$ that has a Heisenberg partition. 
Therefore, by Lemma~\ref{L11}, every coadjoint orbit of $G$ is simply connected. 

On the other hand, since the topological space $\widehat{H}$ is $T_1$, 
it follows that $\widehat{G}$ is $T_1$ as well. 
Hence, by \cite[Thm. 2]{Pu73}, the solvable Lie groups $G$ and $H$ are type~R, 
in the sense that all roots of their Lie algebras are purely imaginary. 
This further implies that all the coadjoint orbits of $G$ and $H$ are closed 
in the dual spaces of the Lie algebras of these groups. 
(See \cite[Prop. 3 and Thm. 2]{Pu78} for groups of type~\I, 
or \cite[page 444]{Pu82} and \cite[page 470]{Pe82} for general groups.)

In particular, all coadjoint orbits of $G$ are closed and simply connected, hence, by  \cite[Thm. 1.1]{BB_RIMS}, the Lie group $G$ is nilpotent. 
This further implies by the criterion of P. Hall 
(see e.g., \cite[Lemma 3.2]{BB_RIMS}) that the group $H$ is nilpotent. 
\end{proof}

We are now in a position to prove Proposition~\ref{P13}. 

\begin{proof}[Proof of Proposition~\ref{P13}]
The hypothesis  $\widehat{H}\simeq\widehat{H_n}$ shows that $\widehat{H}$ has a Heisenberg partition. 
Then, by Lemma~\ref{L12}, the Lie group $H$ is nilpotent. 
Now, by the dual rigidity property of the Heisenberg group within the class of nilpotent Lie groups established in \cite[Thm. 1.2]{BB-JTA}, we obtain a Lie group isomorphism $H\simeq H_n$. 
\end{proof}

The statement of Proposition~\ref{P13} may not hold true if we replace the Heisenberg group by another nilpotent Lie group. 
Specifically, 
we construct below an uncountable family of 1-connected nilpotent Lie groups 
that are pairwise non-isomorphic and yet their unitary dual spaces are homeomorphic to each other. 

\begin{example}
\label{E14}
	\normalfont
	For all $s,t\in\RR\setminus\{0\}$ let $\gg_0(s,t)$ be the 6-dimensional 2-step nilpotent Lie algebra 
	defined by the commutation relations 
	$$[X_6,X_5]=sX_3,\ [X_6,X_4]=(s+t)X_2,\ [X_5,X_4]=tX_1. $$
	Then
	$$\{\gg_0(s,t)\mid s^2+st+t^2=1,\ 0<t\le 1/\sqrt{3}\}$$
	is a family of isomorphic Lie algebras 
	that are however pairwise non-isomorphic as symplectic Lie algebras 
	with 
	the common symplectic structure 
	$$\omega\colon \gg_0(s,t)\times\gg_0(s,t)\to\RR$$ given by $\omega(X_1,X_6)=\omega(X_2,X_5)=\omega(X_3,X_4)=1$. 
(See \cite[Ex. 5.1]{BBP} and the references therein.)
It then follows by \cite[Thm. 3]{SS78} that if we form the 1-dimensional central extensions $\gg(s,t):=\gg_0(s,t)\dotplus_\omega\RR$ as for instance in \cite[Ex. 6.8]{BB15}, we obtain an uncountable family of pairwise non-isomorphic nilpotent Lie algebras. 
See also the Lie algebras $(147E)$ and $(147E_1)$ in \cite[\S 7.3]{Gn98}. 

On the other hand, every Lie algebra $\gg_0(s,t)$ is isomorphic as Lie algebra to the free 2-step nilpotent Lie algebra $\ng$ of dimension~6, 
and  the nilpotent Lie group $G(s,t)$ has 1-dimensional center while its generic coadjoint orbits are flat. 
It then follows that the unitary dual space $\widehat{G(s,t)}$ 
has a partition $\widehat{G(s,t)}=\Gamma_1\sqcup\Gamma_\ng$ with the following properties: 
\begin{enumerate}[{\rm(a)}]
	\item\label{E14_item1} $\Gamma_1\subseteq \widehat{G(s,t)}$ is an open subset and there is a homeomorphism $\psi_1\colon\RR^\times\to\Gamma_1$. 
	\item\label{E14_item2} There exists a homeomorphism $\psi_\ng\colon\widehat{N}\to\Gamma_\ng$. 
	\item\label{E14_item3} 
	For every subset $A\subseteq\RR^\times$ we have
	\begin{equation*}
		\text{$0\in\RR$ is an accumulation point for $A$}
		\iff \Gamma_\ng\cap \overline{\psi_1(A)}\ne\emptyset \iff \Gamma_\ng\subseteq\overline{\psi_1(A)}.
	\end{equation*}
\end{enumerate}
These properties can be most easily seen in terms of the space of coadjoint orbits $\gg(s,t)^*/G(s,t)$ with its quotient topology, using the Kirillov homeomorphism $\kappa\colon \gg(s,t)^*/G(s,t)\to\widehat{G(s,t)}$. 
In some more detail, denote by $\zg$ the 1-dimensional centre of $\gg(s,t)$, 
and select $0\ne X_0\in\zg$, so $\zg=\RR X_0$.  
The generic coadjoint orbits are the affine hyperplanes 
$$\gg(s,t)^*_r:=\{\xi\in \gg(s,t)^*\mid\langle\xi,X_0\rangle=r\}
\text{ for }r\in\RR^\times,$$ 
and the set of these affine hyperplanes corresponds via $\kappa$ to the above set $\Gamma_1$. 
The complement of the union of these affine hyperplanes is 
$$\gg(s,t)^*_0:=\{\xi\in \gg(s,t)^*\mid\langle\xi,X_0\rangle=0\}
=\{\xi\in \gg(s,t)^*\mid\zg\subseteq\Ker\xi\}=\zg^\perp.$$ 
The set $\gg(s,t)^*_0$  is invariant to the coadjoint action of $G(s,t)$. 
The set $\Gamma_\ng$ corresponds, again via the homeomorphism $\kappa$, to the set of all coadjoint orbits of $G(s,t)$ that are contained in the closed subset $\gg(s,t)^*_0\subseteq\gg(s,t)^*$. 
This directly implies that, if $A\subseteq\RR^\times$ 
then for its corresponding set of affine hyperplanes 
$\gg(s,t)^*_A:=\{\gg(s,t)^*_r\mid r\in A\}\subseteq\gg(s,t)^*/G(s,t)$ 
we have
\begin{align*}
	\text{$0\in\RR$ is an accumulation point for $A$}
	& \iff (\gg(s,t)^*_0/G(s,t))\cap \overline{\gg(s,t)^*_A}\ne\emptyset \\
	&\iff \gg(s,t)^*_0/G(s,t)\subseteq \overline{\gg(s,t)^*_A}.
\end{align*}
Here $\gg(s,t)^*_0/G(s,t)\subseteq\gg(s,t)^*/G(s,t)$ 
is the set of coadjoint orbits of $G$ contained in $\gg(s,t)^*_0$. 
We also recall that, if $G$ is a Lie group and $\Oc,\Oc_1,\Oc_2,\dots\in\gg^*/G$ are mutually different coadjoint orbits of $G$, then $\Oc$ is an accumulation point of the sequence $\{\Oc_n\}_{n\ge 1}$ with respect to the quotient topology of $\gg^*/G$ if and only if there exist $\xi\in\Oc$, subscrips $1\le n_1<n_2<\cdots$ and points $\xi_j\in\Oc_{n_j}$ with $\lim\limits_{j\to\infty}\xi_j=\xi$ in~$\gg^*$. 

Now, taking into account the above properties \eqref{E14_item1}, \eqref{E14_item2}, and \eqref{E14_item3}, it follows by Remark~\ref{L00_rem} and Lemma~\ref{L00} that the unitary dual spaces $\widehat{G(s,t)}$  are pairwise homeomorphic for different values of $(s,t)$. 

See also \cite[Subsect. 7.3]{BB-LMA} for more details on the topology of the unitary dual space~$\widehat{N}$. 
\end{example}

\section{Lie groups with co-compact radical}
\label{sect:cocomp}

This section contains the first of the main steps towards proving Theorem~\ref{dual_main}. 
Specifically, we prove the following fact that essentially shows that 
if the Lie group~$G$ in  Theorem~\ref{dual_main} has co-compact radical, then $G$ may be assumed to be solvable, a case that is taken care of by Proposition~\ref{P13}. 
We note again that in this case the group $G$ is of type~\I. 

\begin{proposition}
	\label{cocomp}
	If $G$ is a 1-connected Lie group with co-compact radical, 
	and  $\widehat{G}$ is homeomorphic to the unitary dual of a Heisenberg group, 
	then the Lie group $G$ is solvable. 
\end{proposition}
 
We now begin the preparation for the proof of Proposition~\ref{cocomp} above, 
and the proof itself will be given at the end of this section. 
It is convenient to introduce the following terminology. 

\begin{definition}
\normalfont 
Let $X$ be a topological space. 
A subset $S\subseteq X$ is called a \emph{limit set} if there exists a net $\{x_i\}_{i\in I}$ in $X$ that converges to every point of $S$.  
The topological space $X$ is  said to be of \emph{1-dimensional type} if it 
has a sequence whose limit set (with its relative topology) is 
homeomorphic to a non-singleton subset of~$\ZZ\times\RR$. 

A finite-dimensional real Lie algebra $\hg$, as well as its corresponding 1-connected Lie group $H$, are said to be of \emph{1-dimensional type} if the topological space 
$\widehat{H}$ 
is of 1-dimensional type. 
\end{definition}

\begin{remark}
\label{one-dim}
\normalfont
If a topological space $X$ has a closed subset that is of 1-dimensional type, 
then the space $X$ is of 1-dimensional type. 

This implies that if $\hg$ is a finite-dimensional Lie algebra with an ideal $\ag$ 
such that the quotient $\hg/\ag$ is of 1-dimensional type, then $\hg$  is of 1-dimensional type. 
In fact, let $H$ be a 1-connected Lie group whose Lie algebra is $\hg$. 
If $A=\langle\exp_H(\ag)\rangle$ is the integral subgroup of $H$ corresponding to $\ag$, then $A$ is a 1-connected closed normal subgroup of $H$ by \cite[Ch. II, \S 6, no. 6, Prop. 14]{Bo06}. 
Then the unitary dual of the quotient group $H/A$ is homeomorphic to a closed subset of the unitary dual of $H$. 
Since the Lie group $H/A$ is 1-connected and its Lie algebra is $\hg/\ag$,  the assertion follows. 
\end{remark}

\begin{example}
\label{co-Ex}
\normalfont 
The Lie algebra of the Euclidean group of the plane  $\CC\rtimes \TT$ is of 1-dimensional type, 
as can be seen from \cite[Ex. 2, page 263]{Fe62}. 

On the other hand, the Heisenberg algebras are not of 1-dimensional type. 
More precisely, if $H$ is a Heisenberg group and $F\subseteq \Prim(H)$ is a closed subset, then $F$ is not of 1-dimensional type. 
This assertion is based on Remark~\ref{partition_rem}\eqref{partition_rem_item1} 
and the fact that if $n\ge 2$ then $\RR^n$ is not homeomorphic to any subset of $\ZZ\times\RR$, which in turn follows from the fact that every connected subset of $\ZZ\times\RR$ is homeomorphic to an interval in $\RR$. 
\end{example}

Let $\kg$ be a finite-dimensional real Lie algebra. 
For later use, we recall from \cite[Prop. 12.1.4]{HN12} that  
$\kg$ is a \emph{compact Lie algebra} if and only if there exists a compact connected Lie group whose Lie algebra is isomorphic to $\kg$. 
A finite-dimensional real Lie algebra is compact if and only if it is the direct product of an abelian Lie algebra and a compact semisimple Lie algebra, 
cf. \cite[Lemma 12.1.2]{HN12}.
If $\kg$ is a semisimple Lie algebra, then compactness of $\kg$ is equivalent to the compactness of every connected Lie group whose Lie algebra is isomorphic to $\kg$ by Weyl's theorem \cite[Thm. 12.1.7]{HN12}. 

\begin{lemma}
\label{co-L2}
Let $\kg$ be a 
semisimple Lie algebra, $\ng$ a nilpotent Lie algebra, and $\alpha\colon\kg\to\Der(\ng)$ be a Lie algebra homomorphism. 
Define $\ng_0:=\ng/[\ng,\ng]$ and 
	$$\alpha_0\colon\kg\to\End(\ng_0),\quad \alpha_0(x)(y+[\ng,\ng]):=\alpha(x)y+[\ng,\ng].
$$
Then $\alpha_0$ is a Lie algebra homomorphism and, 
	if 
moreover $\alpha\ne0$ then $\alpha_0\ne 0$. 
\end{lemma}

\begin{proof}
Since $\kg$ is a semisimple Lie algebra and $[\ng,\ng]$ is invariant to $\alpha(\kg)$, 
there exists a linear subspace $\Dc\subseteq\ng$ satisfying 
\begin{equation}
\label{co-L2_proof_eq1}
\ng=[\ng,\ng]\dotplus\Dc\text{ and }\alpha(\kg)\Dc\subseteq\Dc.
\end{equation}
(See for instance \cite[\S 6, no. 2, Thm. 2]{Bo72}.)
The nilpotent Lie algebra $\ng$ is generated by $\Dc$; see e.g., \cite[\S 4, Exerc. 4, page 119]{Bo72}. 
By the hypothesis $\alpha\ne0$, that is, $\alpha(\kg)\ng\ne0$, we then obtain 
$\alpha(\kg)\Dc\ne0$. 
By \eqref{co-L2_proof_eq1}, it follows that $\alpha(\kg)\ng\not\subseteq[\ng,\ng]$, 
which is equivalent to  $\alpha_0\ne0$. 
\end{proof}

The \emph{nilradical} of a finite-dimensional real Lie algebra $\gg$ 
is its largest nilpotent ideal, cf. \cite[Ch. 1, \S 5, no. 4]{Bo72}. 
Thus, the nilradical may be strictly larger than the nilpotent radical, 
which is equal to $[\gg,\gg]$ when $\gg$ is solvable, 
by \cite[Ch. 1, \S 5, no. 3, Cor. 1]{Bo72}.

	If $G$ is a 1-connected Lie group, the radical $\rg$ of $\gg$ is the greatest solvable ideal of $\gg$. 
	Its corresponding integral subgroup $R=\langle\exp\rg\rangle$  is a closed 1-connected subgroup of $G$ (as in Remark~\ref{one-dim}), 
	called the \emph{radical of $G$}.

\begin{lemma}
	\label{co-L3}
	Let $\kg$ be a 
	semisimple 
	Lie algebra, $\rg$ a solvable Lie algebra with its nilradical $\ng$, and $\alpha\colon\kg\to\Der(\rg)$ be a Lie algebra homomorphism. 
	We define $\rg_0:=\rg/[\ng,\ng]$,  $\ng_0:=\ng/[\ng,\ng]$, and 
	$$\alpha_0\colon\kg\to\End(\rg_0),\quad \alpha(x)(y+[\ng,\ng]):=\alpha(x)y+[\ng,\ng].
	$$
	Then the following assertions hold:
	\begin{enumerate}[{\rm(i)}]
		\item\label{co-L3_item1} 
		The Lie algebra $\rg_0$ is metabelian and its nilradical is $\ng_0$. 
		\item\label{co-L3_item2} 
		We have $[\ng_0,\ng_0]=\{0\}$. 
		\item\label{co-L3_item3} 
		If $\alpha\ne0$, then $\alpha_0\ne0$. 
	\end{enumerate}
\end{lemma}

\begin{proof}
\eqref{co-L3_item1}--\eqref{co-L3_item2}
The equality  $[\ng_0,\ng_0]=\{0\}$ is clear. 
It is noted without proof in \cite[proof of Lemma 23, (A)(a)(1), page 43]{Pu78} 
that $\ng_0$ is the nilradical of $\rg_0$. 
For completeness, let us provide a proof of that fact. Specifically, we must show that if $\ag_0\subseteq\rg_0$ is a nilpotent ideal, then $\ag_0\subseteq\ng_0$. 
Since $\ag_0\subseteq\rg_0=\rg/[\ng,\ng]$ is an ideal, there exists an ideal $\ag\subseteq\rg$ with $[\ng,\ng]\subseteq\ag$ and $\ag_0=\ag/[\ng,\ng]$, 
and it remains to prove that $\ag\subseteq\ng$.  
To this end we note that $\ag+\ng$ is an ideal of $\rg$ with $\ng\subseteq\ag+\ng$. 
Moreover, $(\ag+\ng)/[\ng,\ng]=\ag_0+\ng_0$. 
Since both $\ag_0$ and $\ng_0$ are nilpotent ideals of $\rg_0$, it follows that $\ag_0+\ng_0$ is in turn a  nilpotent ideal (contained in the nilradical of $\rg_0$).  
Thus $\ng$ is a nilpotent ideal of the Lie algebra $\ag+\ng$ such that $(\ag+\ng)/[\ng, \ng]$ is nilpotent. 
Then, by Hall's criterion for Lie algebras 
\cite[Rem. 3.3]{BB_RIMS}, we obtain that $\ag+\ng$ is nilpotent. 
Since $\ag+\ng$ is an ideal of $\rg$, it follows that it is contained in the nilradical $\ng$, that is, $\ag\subseteq\ng$, which completes the proof of the fact that $\ng_0$ is the nilradical of $\rg_0$.

Since the Lie algebra $\rg$ is solvable, we have $[\rg,\rg]\subseteq\ng$, hence $[\rg_0,\rg_0]\subseteq\ng_0$, and this implies that the Lie algebra $\rg_0$ is metabelian.

\eqref{co-L3_item3}
Since the Lie algebra $\rg$ is solvable, for every $\delta\in\Der(\gg)$ we have $\delta(\rg)\subseteq\ng$; 
see e.g., \cite[\S 5, no. 5, Prop. 6]{Bo72}. 
In particular, $\alpha(\kg)\ng\subseteq\ng$. 
Since moreover the Lie algebra $\kg$ is semisimple, there exists a linear subspace $\Ec\subseteq\rg$ satisfying 
$$\rg=\ng\dotplus\Ec
\text{ and }\alpha(\kg)\Ec\subseteq\Ec$$
as in the proof of Lemma~\ref{co-L2}. 
Since $\alpha(\kg)\rg\subseteq\ng$, we must have $\alpha(\kg)\Ec\subseteq\Ec\cap\ng =\{0\}$. 
Then, by the hypothesis $\alpha\ne0$, that is, $\alpha(\kg)\rg\ne0$, 
it follows that  $\alpha(\kg)\ng\ne0$. 
By Lemma~\ref{co-L2}, 
we then obtain~$\alpha_0\ne0$.   
\end{proof}

The following terminology is needed in Lemma~\ref{co-L4}\eqref{co-L4_item2} below: 
a finite-dimensional real Lie algebra $\hg$ is said to be \emph{of type~\I} if 
a 1-connected Lie group whose Lie algebra is isomorphic to $\hg$ is type~\I.  
A class of solvable Lie algebras of type~\I, containing all nilpotent Lie algebras, 
is provided by \cite[Thm.]{Di62}, where it is also mentioned that a solvable Lie algebra $\hg$ may not be type~\I\ in the above sense if it is isomorphic to the Lie algebra of some connected Lie group of type~\I that is not simply connected. 

\begin{lemma}
	\label{co-L4} 
Let $\kg$ be a compact semisimple Lie algebra, $\rg$ a solvable Lie algebra with abelian nilradical $\ng$, and $\alpha\colon \kg \to \Der(\rg)$ be a Lie algebra homomorphism with $\alpha \ne 0$. 
	Define $\gg:=\rg\rtimes_\alpha\kg$. 
		If all the roots of $\rg$ are purely imaginary, 
	then at least one of the following assertions holds:
	\begin{enumerate}[{\rm(i)}]
		\item\label{co-L4_item1} 
		There exist two ideals $\ng^0,\ng^1$ of $\gg$ satisfying 
		$\ng=\ng^0\dotplus\ng^1$, $\alpha(\kg)\ng^0\ne0$, and $\rg/\ng^1$ is a nilpotent Lie algebra. 
		\item\label{co-L4_item2} 
		There exists an ideal $\ag_1$ of $\gg$ with $\ag_1\subseteq\rg$ 
		such that, if we define $\rg_1:=\rg/\ag_1$ and $\ng_1$ is the nilradical of $\rg_1$, then $\rg_1$ is a metabelian Lie algebra of type~\I, $\dim(\rg_1/\ng_1)=1$, and $\alpha(\kg)\rg\not\subseteq\ag_1$. 
	\end{enumerate}
\end{lemma}

\begin{proof}
Since $\ng$ is the nilradical of $\rg$ and $[\ng,\ng]=0$, 
it follows that the adjoint representation $\rho\colon \rg\to \End(\ng)$, 
$\rho(x):=(\ad_\rg x)\vert_\ng$,  
satisfies $[\rg,\rg]\subseteq\ng\subseteq\Ker(\ad)$. 
Therefore, $\rho(\rg)$ is an abelian Lie subalgebra of $\End(\ng)$ and then we have the weight space decomposition as in \cite[\S 2.2]{BB-JFAA} 
(see also \cite{ArCu24}), given by
\begin{equation}
\label{co-L4_proof_eq1}
\ng=\ng^0\oplus\ng^1\text{ and }\ng^1=\bigoplus_{\lambda\in\Lambda_1}\ng^\lambda
\end{equation}
for a finite set of weights $\Lambda_1\subseteq\rg^*$, 
$$\ng^0:=\{y\in\ng\mid (\forall x\in\rg)\quad (\ad_\rg x)^{\dim\ng}y=0 \},$$
while $\ng^1$ is a linear subspace of $\ng$ whose structure of $\RR$-linear space 
extends to a structure of $\CC$-linear subspace satisfying $\rho(x)\vert_{\ng^1}\in\End_\CC(\ng^1)$ and 
$$(\forall \lambda\in\Lambda_1)\quad 
\ng^\lambda:=\{y\in\ng\mid (\forall x\in\rg)\quad 
(\rho(x)-\ie\lambda(x)\1)^{\dim\ng}y=0 \}.$$
On the other hand, as in the proof of Lemma~\ref{co-L3}\eqref{co-L3_item3}, the hypothesis $\alpha\ne0$ implies $\alpha(\kg)\ng\ne0$. 
Since $\alpha(\kg)\ng\subseteq\ng$, we also have 
\begin{equation*}
	(\forall\lambda\in\Lambda_1\cup\{0\})\quad \alpha(\kg)\ng^\lambda\subseteq\ng^\lambda 
\end{equation*}
by \cite[Ch. 1, \S 3, no. 1, Prop. 1]{Bo72}. 

For every $\lambda\in\Lambda_1\cup\{0\}$ we also have $\rho(\rg)\ng^\lambda\subseteq\ng^\lambda$, 
hence $\ng^\lambda$ is an ideal of $\rg\rtimes_\alpha\kg=\gg$. 
It then follows by \eqref{co-L4_proof_eq1} 
that we must be in one of the following two cases: 

\eqref{co-L4_item1}  $\alpha(\kg)\ng^0\ne\{0\}$. 
We note that $\rg/\ng^1$ is a nilpotent Lie algebra since it is a solvable Lie algebra with no nonzero roots. 

\eqref{co-L4_item2}  $\alpha(\kg)\ng^0=\{0\}$ and there exists $\lambda_1\in\Lambda_1$ with $\alpha(\kg)\ng^{\lambda_1}\ne\{0\}$. 
We then define 
$$\ag_1:=\ng^0\oplus\bigoplus_{\lambda\in\Lambda_1\setminus\{\lambda_1\}}\ng^\lambda.$$
Then the Lie algebra $\rg_1=\rg/\ag_1$ is metabelian since $\rg$ is. 
Moreover, $\rg_1$ has only one nonzero root, namely $\ie\widetilde{\lambda_1}$, 
where we have defined $\widetilde{\lambda_1}\colon\rg_1\to\RR$, $x+\ag_1\mapsto \ie\lambda_1(x)$. 
It then directly follows by \cite[Thm.]{Di62} that the solvable Lie algebra $\rg_1$ is type~\I. 
On the other hand, since the nilradical of a solvable Lie algebra is equal to the intersection of the kernels of all roots, 
we obtain $\ng_1=\Ker\widetilde{\lambda_1}$, and in particular $\dim(\rg_1/\ng_1)=1$. 
Finally, the assumption $\alpha(\kg)\ng^{\lambda_1}\ne\{0\}$ 
implies $\alpha(\kg)\rg\not\subseteq\ag_1$. 
\end{proof}

\begin{lemma}
	\label{co-L5}
	Let $\kg$ be a compact semisimple Lie algebra, $\rg$ a solvable Lie algebra with 
	abelian nilradical $\ng$, and 
	$\alpha\colon\kg\to\Der(\rg)$ be a Lie algebra homomorphism 
	with $\alpha\ne0$. 
	If  the roots of $\rg$ are purely imaginary 
	and $\dim(\rg/\ng)=1$, 
	then the Lie algebra $\rg\rtimes_\alpha\kg$ is of 1-dimensional type. 
\end{lemma}

\begin{proof}
	First note that, since $\alpha(\kg)\ng\subseteq\ng$ and $\kg$ is a 
	semisimple 
	Lie algebra, there exists a subalgebra $\ag\subseteq\rg$ satisfying $\alpha(\kg)\ag\subseteq\ag$, $\rg=\ng\dotplus\ag$, and $\dim\ag=1$. 
	(See the proof of Lemma~\ref{co-L2} for a similar situation.) 
	
We define $\gg:=\rg\rtimes_\alpha\kg$. 
Just as in the proof of Lemma~\ref{co-L3}\eqref{co-L3_item3}, 
the conditions $\rg=\ng\dotplus\ag$ and $\alpha(\kg)\ag\subseteq\ag$ actually imply $\alpha(\kg)\ag=\{0\}$ and $\alpha(\kg)\ng\ne\{0\}$. 
In particular, the direct product $\hg:=\ag\times\kg$ is a subalgebra of $\gg$ 
and we have the semidirect product decomposition 
\begin{equation*}
\gg=\ng\rtimes\hg.
\end{equation*}
	We will use the duality pairing $\langle\cdot,\cdot\rangle\colon\ng^*\times\ng\to\RR$. 

The property $\alpha(\kg)\ng\ne\{0\}$ implies that there exists $\xi_1\in\ng^*$ with $\langle\xi_1,\alpha(\kg)\ng\rangle\ne\{0\}$ 
and then the isotropy Lie algebra $\kg(\xi_1):=\{x\in\kg\mid \langle\xi_1,\alpha(x)\ng\rangle=0\}$ satisfies 
\begin{equation}
	\label{co-L5_proof_eq1}
	\kg(\xi_1)\subsetneqq\kg. 
\end{equation}
	For arbitrary $t\in\RR$ we now define $\xi_t:=t\xi_1\in\ng^*$ 
and $\chi_t\in\Hom(N,\TT)$ with 
$(\chi_t)'_\1=\ie\xi_t=\ie t\xi_1\colon\ng\to\ie\RR$. 
In particular, $\chi_0(x)=1$ for every $x\in N$. 

If $t\ne0$, then 
$H(\chi_t) = H(\chi_1)$ and  the hypothesis $[\ng,\ng]=\{0\}$ implies 
\begin{equation}
	\label{co-L5_proof_eq2}
	G(\chi_t)=N\rtimes H(\chi_t)=N\rtimes H(\chi_1)=G(\chi_1)
\end{equation}
and the Lie algebra of that group is 
\begin{equation*}
	\gg(\chi_t)=\ng\rtimes\hg(\xi_t)\subsetneqq\ng\rtimes\hg=\gg
\end{equation*}
where the strict inclusion follows by \eqref{co-L5_proof_eq1}.

	For every $t\in\RR$ we now define 
$$\widetilde{\chi_t}\in\Hom(G(\chi_1),\TT), 
\quad \widetilde{\chi_t}\vert_\ng=\chi_t\text{ and } \widetilde{\chi_t}\vert_{H(\chi_1)}=1$$ 
and the induced representation 
$$\pi_t:=\Ind_{G(\chi_t)}^G(\widetilde{\chi_t}).$$
If $t\ne 0$, then, by \eqref{co-L5_proof_eq2}, we have 
$$\pi_t=\Ind_{G(\chi_1)}^G(\widetilde{\chi_t})$$
and moreover $\pi_t$ is an irreducible representation of $G$ by Mackey's theorem 
\cite[Thm. 8.1]{Ma58}. 
(See also \cite[Ch. III]{AuMo66}.)

Since $\lim\limits_{t\to0}\chi_t=\chi_0$ in $\Hom(N,\TT)$, we have $\lim\limits_{t\to 0}\widetilde{\chi_t}=\widetilde{\chi_0}$ in $\Hom(G(\chi_1),\TT)$ and then, 
it follows by the continuity of induction, cf. \cite[Thm. 4.2, page 260]{Fe62}, 
that 
for $t\to0$ we have $[\pi_t]\to[\pi_0]$ in the space $\Tc^\sim(G)$ of unitary equivalence classes of unitary representations of $G$ on separable Hilbert spaces, 
with respect to the inner hull-kernel topology of $\Tc^\sim(G)$. 
(See the proof of Lemma~\ref{L10} for a similar reasoning.) 
On the other hand, using \eqref{co-L5_proof_eq1} and Lemma~\ref{nonfact}, one can prove that the representation $\pi_0$ of $G$ is not a factor representation. 
More specifically, $\pi_0$ is the so-called quasiregular representation of $G$ on $L^2(G/G(\chi_1))$. 
Since by \eqref{co-L5_proof_eq2} we have $G/G(\chi_1)=(N\rtimes H)/(N\rtimes H(\chi_1))
\simeq H/ H(\chi_1)$,  it suffices to prove that the quasiregular representation $\lambda$ of $H$ on $L^2(H/ H(\chi_1))$ is not a factor representation. 
In fact, by Lemma \ref{nonfact}, the restricted representation $\lambda\vert_K$ is not a factor representation, that is, it generates a von Neumann algebra with nontrivial centre. 
On the other hand,  $H=A\times K$ and the abelian group $A$ is contained in the centre of $H$, hence the centre of $\lambda(K)''$ commutes with $\lambda(A)''$. 
Therefore the centre of $\lambda(K)''$ is contained in the centre of $\lambda(H)''$, and thus $\lambda$ is not a factor representation.

We then obtain that the set 
of limit points of $[\pi_t]$ in $\Tc^\sim(G)$ for $t\to 0$
contains at least two points. 

We now prove that the above set of limit points is homeomorphic to a subset of $\widehat{K}\times\widehat{A}\simeq\ZZ\times\RR$. 
To this end, let $\pi$ be an arbitrary irreducible representation of $G$ with 
$[\pi_t]\to[\pi]$ as $t\to 0$.
It then follows by the continuity of the restriction, 
cf. \cite[Cor., page 371]{Fe60}, 
that 
for $t\to 0$ we have $[\pi_t\vert_N]\to[\pi\vert_N]$ in $\Tc^\sim(N)$.
This implies 
$N\subseteq\Ker\pi$ 
by Lemma~\ref{modelind}\eqref{modelind_item4}
applied for $\Vc:=\ng$, which is abelian by hypothesis. 
 
It now follows by the Lie group isomorphism 
$G/N\simeq H=A\times K$ that the unitary equivalence class $[\pi]$ 
can be regarded as an element of the closed subset 
$\widehat{A\times K}\hookrightarrow\widehat{G}$. 
This completes the proof of that fact that the set of all limit points of $[\pi_t]$ for $t\to 0$ is a homeomorphic to a non-singleton subset of $\widehat{K}\times\widehat{A}\simeq\ZZ\times\RR$, and we are done.
\end{proof}

\begin{proposition}
\label{co-P}
If $G$ is a 1-connected Lie group of type~\I\ with co-compact nontrivial radical 
whose roots are purely imaginary, 
then $\widehat{G}$ is either nonconnected or of 1-dimensional type.  
\end{proposition}

\begin{proof}
If $R$ is the radical of $G$ and $K$ is a Levi subgroup of $G$, 
then $K$ is compact by hypothesis. 
Moreover, $K$ is 1-connected 
and $G$ can be written as the semidirect product decomposition $G=RK$ 
by the global Levi decomposition.
(See for instance \cite[Cor. 5.6.8 and Lemma 13.3.7]{HN12}.) 
If $rk=kr$ for all $r\in R$ and $k\in K$, then we have a homeomorphism 
$\widehat{G}\simeq\widehat{R}\times\widehat{K}$. 
Since $K$ is connected, compact and nontrivial, it follows that $\widehat{K}$ is homeomorphic to $\ZZ$, which shows that $\widehat{G}$ is nonconnected. 

We now assume that the natural action of $K$ on $R$ is nontrivial, 
and we prove that $\widehat{G}$ is of 1-dimensional type. 
To this end it suffices to show that the quotient of $G$ by some closed subgroup is of 1-dimensional type. 
This will be done in two steps.

Step 1: Let $N$ be the nilradical of $R$. 
Then both $N$ and its commutator subgroup $[N,N]$ are 1-connected closed normal subgroups of $G$. 
The quotient group $G/[N,N]$ has its Levi decomposition $G/[N,N]=(R/[N,N])K$, 
whose factors do not commute with each other by Lemma~\ref{co-L3}. 
We will show that $G/[N,N]$ is of 1-dimensional type. 
By the discussion above, we may
replace $G$ by $G/[N,N]$ 
and we thus assume without loss of generality that $N$ is abelian. 

Step 2: 
Since the nilradical $N$ of $R$ is abelian, it follows by Lemma~\ref{co-L4} 
that at least one of the following assertions holds: 
\begin{enumerate}[{\rm(i)}]
	\item\label{co-P_item1} 
	There exists a 1-connected closed normal subgroup $N^1$ of $G$ with $N^1\subseteq N$ such that $R/N^1$ is nilpotent and the factors in the Levi decomposition $G/N^1=(R/N^1)K$ do not commute with each other. 
	\item\label{co-P_item2} 
	There exists a 1-connected closed normal subgroup $A_1$ of $G$ with $A_1\subseteq R$ such that, if we denote $R_1:=R/A_1$ and $N_1$ the nilradical of $R_1$, 
	then $R_1$ is a metabelian Lie group of type~\I, 
	$\dim(R_1/N_1)=1$, and the factors in the Levi decomposition 
	$G/A_1=R_1 K$ do not commute with each other. 
\end{enumerate}
If the above Assertion~\eqref{co-P_item1} holds true, then Lemma~\ref{co-L2} shows that the group $G/N^1$ is of 1-dimensional type, and we are done. 
On the other hand, if Assertion~\eqref{co-P_item2} in Step~2 holds true, 
then the group $G/A_1$ is of 1-dimensional type by Lemma~\ref{co-L5}, and this completes the proof. 
\end{proof}

\begin{proof}[Proof of Proposition~\ref{cocomp}]
Use Proposition~\ref{co-P} and Example~\ref{co-Ex}. 
\end{proof}

\section{Proof of the main theorem and open problems}
\label{sect:proof}

\begin{proof}[Proof of Theorem~\ref{dual_main}]
We organize the proof in four steps. 

Step 1: Since $H$ is a liminary group, it follows by\cite[9.5.3]{Di64} that $G$ is in turn liminary, in particular type \I\ and $\widehat{G}$ is $T_1$, hence $G=G_1\times (S\rtimes K)$, where $G_1$ is semisimple, $S$ is solvable, and $K$ is compact semisimple.  
(See \cite[proof of Prop. 4.7]{BB_RIMS}.)

Step 2: We then obtain $\widehat{G_1}\times\widehat{S\rtimes K}=\widehat{G}\simeq\widehat{H}$, hence, by \cite[Lemma 4.6]{BB_RIMS}, 
we have either $G=G_1$ or $G=S\rtimes K$. 

Step 3: If $G=G_1$ is semisimple, then by \cite[Cor. 5.8]{Mi74} 
every net in $\widehat{G}$ has at most finitely many limit points. 
(See also \cite{Mi73}.)
Therefore we cannot have $\widehat{G}\simeq \widehat{H}$.

Step 4: If  $G=S\rtimes K$, then by Proposition~\ref{cocomp} we obtain $G=S$, and then by Proposition~\ref{P13} we finally obtain  $G\simeq H$. 
\end{proof}

The natural question arises about finding versions of Theorem~\ref{dual_main} 
with the Heisenberg groups replaced by other 1-connected Lie groups. 
Example~\ref{E14} shows a class of examples of nilpotent Lie groups for which the corresponding versions of Theorem~\ref{dual_main} fail to be true. 

On the other hand, we proved in \cite{BB-JTA} that there are many nilpotent Lie groups that are uniquely determined by the (Morita-)isomorphism class of their group $C^*$-algebras. 
We are thus lead to raising the following problem. 

\begin{problem}
\normalfont 
Do there exist nilpotent Lie groups $G_1$ and $G_2$ with $\widehat{G_1}\simeq\widehat{G_2}$, for which $C^*(G_1)$ and $C^*(G_2)$ fail to be Morita equivalent/$*$-isomorphic? 
\end{problem}

We also refer to \cite{BB-IEOT} for a discussion of pairs of  1-connected solvable Lie groups that are non-isomorphic although their group $C^*$-algebras are $*$-isomorphic. 

\subsection*{Acknowledgment} 
We wish to thank Professor Dragan Mili\v ci\'c for several illuminating remarks and for kindly sending us the preprint \cite{Mi73}.

\end{document}